%% file: gamma_arXiv.tex
\documentclass[12pt]{article}

\usepackage{graphicx}
\usepackage{amsmath}
\usepackage{amssymb}
\usepackage{latexsym}
\usepackage{subfigure}
\usepackage{crop}
\usepackage{algorithmic}
\usepackage{algorithm}
\usepackage{multirow}
\usepackage[section,subsection,subsubsection]{placeins}
\usepackage{bm}
\usepackage{enumerate}
\usepackage[framed,numbered,autolinebreaks]{mcode}
\usepackage{epstopdf}
\usepackage[colorlinks = true, pdfstartview = FitV, linkcolor = blue, citecolor = blue, urlcolor = blue]{hyperref}

\newcommand {\uu}  { {\bf u} }
\newcommand {\zz}  { {\bf z} }

\newcommand {\xx}  { {\bf x} }

\newcommand{\R}{{\rm I\!R}}

\newcommand {\veps} {\varepsilon}

\newcommand {\Ex} { {\mathbb E} }

\newcommand {\qq}  { {\bf q} }

\newcommand {\mm}  { {\bf m} }

\newcommand {\ww}  { {\bf w} }
\newcommand {\ff}  { {\bf f} }

\newcommand {\dd}  { {\bf d} }

\newcommand{\hf}{\frac12}

\renewcommand{\vec}[1]{\ensuremath{\mathbf{#1}}}

\newcommand{\grad}{\ensuremath {\vec \nabla}}

\renewcommand{\div}{\nabla\cdot\,}

\newcommand{\defeq}{\mathrel{\mathop:}=}

\renewcommand{\thesubfigure}{(\alph{subfigure})} 

\newtheorem{theorem}{Theorem}

\newtheorem{corollary}[theorem]{Corollary}

\newtheorem{lemma}[theorem]{Lemma}

\newenvironment{proof}[1][Proof]{\begin{trivlist}
\item[\hskip \labelsep {\bfseries #1}]}{\end{trivlist}}

\newcommand{\qed}{\nobreak \ifvmode \relax \else
      \ifdim\lastskip<1.5em \hskip-\lastskip
      \hskip1.5em plus0em minus0.5em \fi \nobreak
      \vrule height0.75em width0.5em depth0.25em\fi}
      
\begin{document}

\title{Assessing stochastic algorithms for large scale nonlinear least squares 
problems using extremal probabilities of linear combinations of gamma random variables} 
\author{Farbod Roosta-Khorasani\thanks{Dept. of Computer Science, University of British Columbia, Vancouver, Canada.
{\tt farbod/ascher@cs.ubc.ca}. The work of these authors was partially funded by NSERC grant 84306.} \and G\'{a}bor J. Sz\'{e}kely\thanks{National Science Foundation, Arlington, Virginia. {\tt gszekely@nsf.gov} and Alfr\'{e}d R\'{e}nyi Institute of Mathematics, Hungarian Academy of Sciences,  Budapest, Hungary.} \and Uri M. Ascher\footnotemark[1]}
\maketitle
\begin{abstract}
This article considers stochastic algorithms for efficiently solving a class of 
large scale non-linear least squares (NLS) problems which frequently arise in applications. 
We propose eight variants of a practical randomized algorithm where the uncertainties in the major 
stochastic steps are quantified. Such stochastic steps involve approximating the NLS 
objective function using Monte-Carlo methods, and this is
equivalent to the estimation of the trace of corresponding symmetric positive semi-definite (SPSD) matrices. 
For the latter, we prove \textit{tight} \textit{necessary} and \textit{sufficient} conditions on the 
sample size (which translates to cost) 
to satisfy the prescribed probabilistic accuracy. We show that these conditions are practically computable 
and yield small sample sizes. 
They are then incorporated in our stochastic algorithm to quantify the uncertainty in each randomized step. 
The bounds we use
are applications of more general results regarding extremal tail probabilities of linear combinations of gamma distributed random variables. We derive and prove new results concerning the maximal and 
minimal tail probabilities of such linear combinations, which can be considered 
independently of the rest of this paper.
\end{abstract}


%

\pagestyle{myheadings}
\thispagestyle{plain}



\section{Introduction}
\label{sec:intro}
Large scale data fitting problems arise often in many applications in science and engineering. 
As the ability to gather larger amounts of data increases, 
the need to devise algorithms to efficiently solve such problems becomes more important. 
The main objective here is typically to recover
some model parameters,
and it is a widely accepted working assumption that having more data can only help 
(at worst not hurt) the model recovery.

Consider the system\footnote{In this paper, we use bold lower case to denote vectors
and regular capital letters to denote matrices.}
\begin{equation}
\dd_{i}  = \ff_{i}(\mm) + \bm{\eta}_{i}, \; i = 1,2, \ldots , s,	
\label{general_forward_problem}
\end{equation}
where $\dd_{i} \in \mathbb{R}^{l}$ is the measurement data obtained in the $i^{th}$ experiment, 
$\ff_{i} = \ff_i(\mm) $ is the known forward operator (or data predictor) for the $i^{th}$ experiment, 
$\mm \in \mathbb{R}^{l_{m}}$ is the sought-after parameter vector\footnote{
The parameter vector $\mm$ often arises from a parameter function in several space variables projected onto a discrete grid and reshaped into a vector.}, 
and $\bm{\eta}_{i}$ is the noise incurred in the $i^{th}$ experiment. 
The total number of experiments, or data sets, is assumed large: $s \gg 1$.
The goal is to find (or infer) the unknown model, $\mm$, from the measurements $\dd_i$, $i = 1, 2, \ldots , s$. 
Generally, this problem 
can be ill-posed. 
Various approaches, including different regularization techniques, 
have been proposed to alleviate this ill-posedness; see, e.g.,~\cite{vogelbook,ehn1}.

In this paper we assume that
the forward operators have the form 
\begin{equation}
\ff_{i}(\mm) = \ff(\mm,\qq_{i}), \quad i = 1, \ldots , s, 
\label{forward_op}
\end{equation}
where $\qq_{i}$ are inputs 
such that the $i^{th}$ data set, $\dd_{i}$, is measured after injecting the 
$i^{th}$ input (or source) $\qq_{i}$ into the system. 
Thus, for an input $\qq_{i}$, $\ff(\mm,\qq_{i})$ predicts the $i^{th}$ measurement, 
given the underlying model $\mm$. 
We only consider a special case where $\qq_{i} \in \mathbb{R}^{l_{q}}, \forall i$, and $\ff$
is linear in $\qq$, i.e.,
$\ff(\mm,w_{1} \qq_{1} + w_{2} \qq_{2}) = w_{1} \ff(\mm,\qq_{1}) + w_{2} \ff(\mm,\qq_{2})$. 
Alternatively, we write $\ff(\mm,\qq) = G(\mm)\qq$, 
where $G \in \mathbb{R}^{l \times l_{q}}$ is a matrix that depends non-linearly on the sought $\mm$. 
%
We also assume that the task of evaluating $\ff$ for each input, $\qq_{i}$, is computationally expensive. 
Examples of such a situation arise frequently in PDE constrained inverse problems with many data sets; 
see, e.g.,~\cite{HaberChungHermann2010,doas3,rodoas1} and references therein. 


Under the further assumption that the independent noise satisfies\footnote{
For notational simplicity, we do not distinguish between a random vector (e.g., noise) and its realization, as they are clear within the context in which they are used.}
$\bm{\eta}_{i} \sim \mathcal{N}(0,\sigma \mathbb{I}), \forall i$, where $\mathcal{N}$ denotes normal distribution, 
$\mathbb{I} \in \mathbb{R}^{l \times l}$ denotes the identity matrix and $\sigma > 0$, 
the standard \textit{maximum likelihood} (ML) approach leads to minimizing the $\ell_2$ misfit function 
\begin{equation}
\phi(\mm) \defeq \sum_{i=1}^{s} \| \ff(\mm,\qq_{i}) - \dd_{i} \|_{2}^{2}.
\label{misfit}
\end{equation}
However, since the above inverse problem is typically ill-posed, 
a regularization functional, $R(\mm)$, is often added to the above objective, thus minimizing instead
\begin{equation} 
\phi_{R,\alpha}(\mm) \defeq  \phi(\mm) + \alpha R(\mm) ,
\label{objective}
\end{equation}
where $\alpha$ is a regularization parameter~\cite{ehn1}. 
In general, this regularization term can be chosen using a priori knowledge of the desired model. 
The objective functional~\eqref{objective} coincides with the \textit{maximum a posteriori (MAP)} formulation.  
Implicit regularization 
also exists in which there is no explicit term $R(\mm)$ in the objective~\cite{hansen1998,doas3}. 
Various optimization techniques can be used to decrease the value of the above objective functionals,~\eqref{misfit} 
or~\eqref{objective}, to a desired level (determined, e.g., by a given tolerance which depends on the noise level), 
thus recovering the sought-after model.

Algorithms that rely on efficiently approximating the misfit function $\phi(\mm)$ have been proposed
and studied in~\cite{HaberChungHermann2010, doas3, rodoas1, rodoas2, roas1}. In effect,
they draw upon estimating the trace of an implicit\footnote{By ``implicit matrix'' we mean that 
the matrix of interest is not available explicitly: only information in the form of matrix-vector products for any appropriate vector is available.} symmetric positive semi-definite (SPSD) matrix. 
To see this, rewrite~\eqref{misfit} as
\begin{equation}
\phi(\mm) = \| F(\mm) - D \|_{F}^{2},
\label{misfit_mat}
\end{equation}
where $F(\mm)$ and D are $l \times s$ matrices whose $i^{th}$ columns are, respectively, $\ff(\mm,\qq_{i})$ and $\dd_{i}$,
and $\| \cdot \|_F$ stands for the Frobenius norm. 
Now, letting $B = B(\mm) \defeq F(\mm) - D$, it can be shown that 
\begin{equation}
\phi(\mm) = \| B \|_{F}^{2} = tr (B^{T} B) = \Ex(\| B \ww \|_{2}^{2}),
\label{frob_trace}
\end{equation}
where $\ww$ is a random vector drawn from any distribution satisfying $\Ex(\ww \ww^{T}) = \mathbb{I}$, 
$tr(A)$ denotes the trace of the matrix $A$, and $\Ex$ denotes the expectation. 
Hence, approximating the misfit function~$\phi(\mm)$ in~\eqref{misfit} or  
in~\eqref{objective} is equivalent to approximating the corresponding matrix trace (or equivalently, approximating the above expectation). 
The standard approach for doing this is based on a Monte-Carlo method, 
where one generates $n$ random vector realizations, $\ww_{j}$, from a suitable probability distribution 
and computes the empirical mean
\begin{equation}
\widehat{\phi} (\mm,n) \defeq \frac{1}{n} \sum_{j=1}^{n} \|B(\mm) \ww_{j}\|_{2}^{2} \approx \phi(\mm).
\label{approx_phi}
\end{equation}
Note that $\widehat{\phi}(\mm,n)$ is an {\em unbiased estimator} of $\phi(\mm)$, as we have 
$\phi(\mm) = \Ex (\widehat{\phi}(\mm,n))$.
For the special case of the forward operators \eqref{forward_op} considered in this paper,
if $n \ll s$ then 
this procedure yields a very efficient algorithm for approximating the misfit~\eqref{misfit}, because 
\begin{equation*}
\sum_{i=1}^{s} \ff(\mm,\qq_{i}) w_{i}  = \ff(\mm,\sum_{i=1}^{s} \qq_{i} w_{i}) ,
\end{equation*}
which can be computed with a single evaluation of $\ff$ per realization of the random vector 
$\ww = (w_1, \ldots , w_s)^T$.


Our assumption regarding the noise distribution leading to the ordinary least squares misfit function~\eqref{misfit},
although standard, is quite simplistic. Fortunately, however, it can be readily generalized in one
of the following two ways.
\begin{enumerate}
	\item 
The noise is independent and identically distributed (i.i.d) as 
	$\bm{\eta}_{i} \sim \mathcal{N}(0,\Sigma ), \forall i$, where $\Sigma \in \mathbb{R}^{l \times l}$ is the 
	symmetric positive definite covariance matrix. 
In this case, 
the ML approach leads to minimizing the $\ell_2$ misfit function 
\begin{equation}
\phi_{(1)}(\mm) \defeq \sum_{i=1}^{s} \| C^{-1} \big(\ff(\mm,\qq_{i}) - \dd_{i} \big)\|_{2}^{2},
\label{misfit_iida}
\end{equation}
where $C \in \mathbb{R}^{l \times l}$ is any invertible matrix such that $\Sigma = C C^{T}$ 
(e.g., $C$ can be the Cholesky factor of $\Sigma$).
Thus,	
\begin{equation*}
\phi_{(1)}(\mm) = \| C^{-1} \big( F(\mm) - D \big) \|_{F}^{2} = \| B(\mm) \|_{F}^{2},
\end{equation*}
with $B(\mm) \defeq C^{-1} \big( F(\mm) - D \big)$.
The Monte-Carlo approximation $\widehat{\phi}_{(1)}(\mm,n)$ is then precisely
as in~\eqref{approx_phi} but with the newly defined $B(\mm)$.
\item 
The noise is independent but \textit{not} identically distributed,
	satisfying instead $\bm{\eta}_{i} \sim \mathcal{N}(0,\sigma^{2}_{i} \mathbb{I}), i = 1,2,\ldots,s$, where
	$\sigma_{i} > 0$ are the standard deviations. 
	Under this assumption, the ML approach yields the \textit{weighted least squares} misfit function
\begin{equation}
\phi_{(2)}(\mm) \defeq \sum_{i=1}^{s} \frac{1}{\sigma^{2}_{i}}\| \ff(\mm,\qq_{i}) - \dd_{i} \|_{2}^{2}.
\label{misfit_inid}
\end{equation}
We can further write this equation as
\begin{equation*}
\phi_{(2)}(\mm) = \| \big( F(\mm) - D \big) C^{-1}\|_{F}^{2},
\end{equation*}
where $C \in \mathbb{R}^{s \times s}$ denotes the diagonal matrix whose $i^{th}$ diagonal element is $\sigma_{i}$.
Thus, with $B(\mm) = ( F(\mm) - D) C^{-1}$ we can again apply~\eqref{approx_phi}
to obtain a similar Monte-Carlo approximation $\widehat{\phi}_{(2)}(\mm,n)$. 
\end{enumerate}
Now, if $n \ll s$ then the unbiased estimators $\widehat{\phi}_{(1)}(\mm,n)$ and $\widehat{\phi}_{(2)}(\mm,n)$
are obtained with a similar efficiency as $\widehat{\phi} (\mm,n)$.
In the sequel, for notational simplicity, we just concentrate on $\phi(\mm)$ and $\widehat{\phi}(\mm,n)$, but all the results 
hold almost verbatim also for~\eqref{misfit_iida} and~\eqref{misfit_inid}.


Hence, the objective is to be able to generate as few realizations of $\ww$ as possible for
achieving acceptable approximations to the misfit function. Estimates on how large $n$ must be
to achieve a prescribed accuracy in a probabilistic sense have been derived in
\cite{achlioptas,avto,yori,roas1}.
However, the obtained bounds are typically not sufficiently tight to be practically useful.
In the present paper, we prove \textit{tight} bounds for tail probabilities for such 
Monte-Carlo approximations employing 
the standard normal distribution. 
These tail bounds are then used to obtain \textit{necessary} and \textit{sufficient} bounds on $n$, 
and we demonstrate that these bounds can be practically small and computable. 
Furthermore, using these results, we are able to better quantify the uncertainties in the highly efficient randomized algorithms proposed in~\cite{doas3,rodoas1,rodoas2}. 
Variants of such algorithms with better uncertainty quantification are derived.

This paper is organized as follows. 
In Section~\ref{sec:trace}, we develop and state theorems regarding the tight tail bounds promised above.
The theory in this section relies upon some novel results regarding the extremal probabilities 
(i.e., maxima and minima of the tail probabilities) of non-negative linear combinations of gamma random variables,
which are proved in Appendix~\ref{sec:appendix}. 

In Section~\ref{sec:alg_intro} we present our stochastic algorithm variants for approximately minimizing 
\eqref{misfit} or~\eqref{objective} and discuss its novel elements. 
Subsequently in Section~\ref{sec:application}, the efficiency of the proposed algorithm variants is demonstrated using an important class of problems that arise often in practice. This is followed by conclusions and further thoughts in Section~\ref{sec:conc}. 


  
\section{Matrix trace estimation}
\label{sec:trace}
Let the matrix $A=B^TB \in \mathbb{R}^{s \times s}$ be implicit SPSD, and denote its trace by $tr(A)$. 
As described in Section~\ref{sec:intro},  we approximate $tr(A)$ by
\begin{equation}
tr_{n}(A) \defeq \frac{1}{n} \sum_{j=1}^{n} \ww_{j}^{T} A \ww_{j},
\label{trace_approx}
\end{equation}
where $\ww_{j} \in \mathbb{R}^{s} \sim \mathcal{N}(0,\mathbb{I})$.

Now, given a pair of small positive real numbers $(\veps,\delta)$, 
consider finding an appropriate sample size $n$ such that
\begin{subequations}
\begin{eqnarray}
\Pr\Big( tr_{n}(A) \geq (1-\veps) tr(A) \Big) \geq 1-\delta \label{prob_ineq_lower}, \\
\Pr\Big( tr_{n}(A) \leq (1+\veps) tr(A) \Big) \geq 1-\delta \label{prob_ineq_upper}.
\end{eqnarray}
\label{prob_ineq_lower_upper}
\end{subequations}
In~\cite{roas1} we showed that the inequalities~\eqref{prob_ineq_lower_upper}
hold if 
\begin{equation}
n > 8c, \quad {\rm where~} c = c(\veps,\delta) = \veps^{-2}\ln (1 / \delta).
\label{loose_bnd}
\end{equation}
However, this bound on $n$ can be rather pessimistic. 
Theorems~\ref{main_trace_theorem_lower} and~\ref{main_trace_theorem_upper} and Corollary
\ref{main_trace_corollary} below provide
tighter and hopefully more useful bounds on $n$.
In order to prove these we require the two additional Theorems~\ref{monotonicity_gamma_theorem} and
\ref{extremal_prob_thm}, whose nontrivial and more technical
proofs are deferred to Appendix~\ref{sec:appendix}. Let $X \sim Gamma(\alpha,\beta)$ denote a gamma distributed random variable (r.v) parametrized by shape $\alpha$ and rate $\beta$\footnote{Recall that the probability density function of such r.v is 
\begin{eqnarray*} 
f(x) = \begin{cases}
\frac{\beta^{\alpha}}{\Gamma(\alpha)} x^{\alpha - 1 } e^{-\beta x}  &\text{$x \ge 0$} , \cr
0 &\text{$x < 0$} .\end{cases} 
\end{eqnarray*}}.
\input exline
\begin{theorem}[\textbf{Monotonicity of cumulative distribution function of gamma r.v}]
Given parameters $ 0 < \alpha_{1} < \alpha_{2}$, let $X_{i} \sim	Gamma(\alpha_{i},\alpha_{i}),\; i=1,2$, be independent r.v's, and define $$\Delta(x) \defeq \Pr(X_{2} < x) - \Pr(X_{1} < x). \label{diff_prob_ga}$$ Then we have that
\begin{enumerate}[(i)]
		\item there is a unique point $x(\alpha_{1},\alpha_{2})$ such that $\Delta(x) < 0$ for $0 < x < x(\alpha_{1},\alpha_{2})$ and $\Delta(x) > 0$ for $x > x(\alpha_{1},\alpha_{2})$, 
		\item $1 \leq x(\alpha_{1},\alpha_{2}) \leq \frac {2  \sqrt{\alpha_{1} ( \alpha_{2} - \alpha_{1}) } + 1}{2 \sqrt{\alpha_{1} ( \alpha_{2} - \alpha_{1}) } }$.
\end{enumerate}
\label{monotonicity_gamma_theorem}
\end{theorem}
\input exline
\begin{theorem}[\textbf{Extremal probabilities of linear combination of gamma r.v's}]
Given shape and rate parameters $\alpha , \beta > 0$, 
let $X_{i} \sim	Gamma(\alpha,\beta),\; i=1,2,\ldots,n$, be i.i.d gamma r.v's, and define $$\Theta \defeq \{ \bm{\lambda} = \big(\lambda_{1}, \lambda_{2}, \ldots, \lambda_{n}\big)^T \; | \; \lambda_i \ge 0 \; \forall i, \; \sum_{i=1}^{n} \lambda_{i} = 1\}.$$
Then we have
\begin{eqnarray*}
m_{n}(x) &\defeq& \min_{\bm{\lambda} \in \Theta} \Pr \left(\sum_{i=1}^{n}   \lambda_{i}  X_{i} < x\right)
=\begin{cases}
\Pr \Big(\frac{1}{n}\sum_{i=1}^{n} X_{i} < x \Big),~ &  x < \frac{\alpha}{\beta} \cr
\Pr \Big(X_{1} < x \Big),~&   x > \frac{2\alpha+1}{2\beta} ,\end{cases} \\
M_{n}(x) &\defeq&  \max_{\bm{\lambda} \in \Theta} \Pr \left(\sum_{i=1}^{n}   \lambda_{i}  X_{i} < x\right)
= \begin{cases} \Pr \Big(X_{1} <  x \Big),~& x < \frac{\alpha}{\beta}  \cr
	\Pr \Big(\frac{1}{n}\sum_{i=1}^{n} X_{i} < x \Big),~& x > \frac{2\alpha+1}{2\beta}
	\end{cases} .  
	\end{eqnarray*}
\label{extremal_prob_thm}
\end{theorem}



Next we state and prove the results that are directly relevant to this section. 
Let us define
\begin{equation*}
Q(n) \defeq \frac{1}{n} Q_{n}, 
\label{Q_n}
\end{equation*}
where $Q_{n} \sim \chi^{2}_{n}$ denotes a chi-squared r.v of degree $n$. Note that $Q(n) \sim Gamma(n/2,n/2)$.
In case of several i.i.d gamma r.v's of this sort, we refer to the $j^{th}$ r.v by $Q_{j}(n)$.

\begin{theorem}[\textbf{Necessary and sufficient condition for~\eqref{prob_ineq_lower}}]
Given an SPSD matrix $A$ of rank $r$ and tolerances $(\veps,\delta)$ as above, the following hold:
\begin{enumerate}[(i)]
	\item \textbf{Sufficient condition:} there exists some integer $n_{0} \geq 1$ such that
	\begin{equation}
			\Pr\big(Q(n_{0}) < (1-\veps) \big) \leq \delta .
		\label{max_bnd_N_lower}
	\end{equation}
Furthermore,~\eqref{prob_ineq_lower} holds for all $n \geq n_{0}$.
	\item \textbf{Necessary condition:} if ~\eqref{prob_ineq_lower} holds for some $n_{0} \geq 1$, 
	then for all $n \geq n_{0}$
	\begin{equation}
			P^{-}_{\veps,r}(n) \defeq \Pr\big(Q(n r) < (1-\veps)\big) \leq \delta.
		\label{min_bnd_N_lower}
	\end{equation} 
	\item \textbf{Tightness:} if the $r$ positive eigenvalues of $A$ are all equal
	 (NB this always happens if $r=1$), then
	there is a positive integer $n_{0}$ satisfying~\eqref{min_bnd_N_lower}, such that~\eqref{prob_ineq_lower} holds iff $n \geq n_0$.
\end{enumerate}
\label{main_trace_theorem_lower}
\end{theorem}

\begin{proof}
Since $A$ is SPSD, it can be diagonalized by a unitary similarity transformation as $A = U^{T} \Lambda U$,
where $\Lambda$ is the diagonal matrix of eigenvalues sorted in non-increasing order.
Consider $n$ random vectors 
$\ww_{i}, \; i = 1, \ldots , n$, whose components are i.i.d and drawn from the 
standard normal distribution, and define $\zz_{i} = U \ww_{i}$ for each $i$. 
Note that since $U$ is unitary, the entries of $\zz_{i}$ are i.i.d standard normal variables, 
like the entries of $\ww_{i}$. We have
\begin{eqnarray*}
\frac{tr_{n}(A)}{tr(A)}  &=& \frac{1}{n~tr(A)} \sum_{i=1}^{n} \ww_{i}^{T} A \ww_{i}  = 
\frac{1}{n~tr(A)} \sum_{i=1}^{n} \zz_{i}^{T} \Lambda \zz_{i} = 
\frac{1}{n~tr(A)} \sum_{i=1}^{n} \sum_{j=1}^{r} \lambda_{j} z_{ij}^{2} \nonumber \\
&=& \sum_{j=1}^{r} \frac{\lambda_{j}}{n~tr(A)} \sum_{i=1}^{n} z_{ij}^{2} 
= \sum_{j=1}^{r} \frac{\lambda_{j}}{tr(A)} Q_{j}(n) ,
\end{eqnarray*}
where the $\lambda_{j}$'s appearing in the sums are positive eigenvalues of A. Now, noting that 
$\sum_{j=1}^{r} \frac{\lambda_{j}}{tr(A)} = 1$,
Theorem~\ref{extremal_prob_thm} yields

\begin{subequations}
\begin{eqnarray}
\Pr\left(\sum_{j=1}^{r} \frac{\lambda_{j}}{tr(A)} Q_{j}(n) \leq (1-\veps)\right) &\leq&  \Pr\big( Q(n) \leq (1-\veps)\big) = P^{-}_{\veps,1}(n) , \label{max_bnd_N_lower_proof}\\
\Pr\left(\sum_{j=1}^{r} \frac{\lambda_{j}}{tr(A)} Q_{j}(n) \leq (1-\veps)\right) &\geq& \Pr\big( Q(n r) \leq (1-\veps) \big) = P^{-}_{\veps,r}(n). \label{min_bnd_N_lower_proof}
\end{eqnarray}
\label{min_max_bnd_N_lower_proof}
\end{subequations}
In addition, for any given $r > 0$ and $\veps > 0$, the function $P^{-}_{\veps,r}(n)$ 
is monotonically decreasing on integers $n \geq 1$. This can be seen by 
Theorem~\ref{monotonicity_gamma_theorem} using the sequence $\alpha_{i} = (n_{0}+(i-1))r/2, \; i \geq 1$. The claims now easily follow by combining~\eqref{min_max_bnd_N_lower_proof} 
and this decreasing property. 
$\blacksquare$
\end{proof}

\begin{theorem}[\textbf{Necessary and sufficient condition for~\eqref{prob_ineq_upper}}]
Given an SPSD matrix $A$ of rank $r$ and tolerances $(\veps,\delta)$ as above, the following hold:
\begin{enumerate}[(i)]
		\item \textbf{Sufficient condition:} if the inequality 
	\begin{equation}
			\Pr\big(Q(n_{0}) \leq (1+\veps) \big) \geq 1-\delta
		\label{max_bnd_N_upper}
	\end{equation}
	is satisfied for some $n_{0} > \veps^{-1}$, then~\eqref{prob_ineq_upper} holds with $n = n_{0}$. 
	Furthermore, there is always an $n_{0} > \veps^{-2}$ such that \eqref{max_bnd_N_upper}
	is satisfied and, for such $n_{0}$, it follows that~\eqref{prob_ineq_upper} holds for all $n \geq n_{0}$.
	\item \textbf{Necessary condition:} if~\eqref{prob_ineq_upper} 
	holds for some $n_{0} > \veps^{-1}$, then
	\begin{equation}
			P^{+}_{\veps,r}(n) \defeq \Pr\big(Q(n r) \leq (1+\veps)\big) \geq 1- \delta,
		\label{min_bnd_N_upper}
	\end{equation}
	with $n = n_{0}$. Furthermore, if $n_{0} > \veps^{-2} r^{-2}$, 
	then ~\eqref{min_bnd_N_upper} holds for all $n \geq n_{0}$.
	\item \textbf{Tightness:} if the $r$ positive eigenvalues of $A$ are all equal, then
	there is a smallest $n_{0} > \veps^{-2} r^{-2}$ satisfying~\eqref{min_bnd_N_upper}  
	such that for any $n \geq n_{0}$,~\eqref{prob_ineq_upper} holds, 
	and for any $\veps^{2} r^{-2} < n < n_{0}$,~\eqref{prob_ineq_upper} does not hold. 
	If $\delta$ is small enough so that~\eqref{min_bnd_N_upper} does not hold for any 
	$n \leq \veps^{2} r^{-2}$, then $n_{0}$ is both necessary and sufficient for~\eqref{prob_ineq_upper}. 
\end{enumerate}
\label{main_trace_theorem_upper}
\end{theorem}

\begin{proof}
The same unitary diagonalization argument as in the proof of Theorem~\ref{main_trace_theorem_lower} shows that 
\begin{equation*}
\Pr\Big( tr_{n}(A) < (1+\veps) tr(A) \Big) = \Pr\left(\sum_{j=1}^{r} \frac{\lambda_{j}}{tr(A)} 
Q_{j}(n) < (1+\veps)\right). 
\end{equation*}
Now we see that if $n > \veps^{-1}$, Theorem~\ref{extremal_prob_thm} with $\alpha = n/2$ yields 
\begin{subequations}
\begin{eqnarray}
\Pr\left(\sum_{j=1}^{r} \frac{\lambda_{j}}{tr(A)} Q_{j}(n) \leq (1+\veps)\right) &\geq&  
\Pr\big( Q(n) \leq (1+\veps)\big) = P^{+}_{\veps,1}(n), \label{max_bnd_N_upper_proof}\\
\Pr\left(\sum_{j=1}^{r} \frac{\lambda_{j}}{tr(A)} Q_{j}(n) \leq (1+\veps)\right) &\leq& 
\Pr\big( Q(n r) \leq (1+\veps) \big) = P^{+}_{\veps,r}(n). \label{min_bnd_N_upper_proof}
\end{eqnarray}
\label{max_min_bnd_N_upper_proof}
\end{subequations}
In addition, for any given $r > 0$ and $\veps > 0$, the function $P^{+}_{\veps,r}(n)$
is monotonically increasing on integers $n > \veps^{-2} r^{-2}$. 
This can be seen by Theorem~\ref{monotonicity_gamma_theorem} using the sequence 
$\alpha_{i} = (n_{0}+(i-1))r/2, \; i \geq 1$. 
The claims now easily follow by combining~\eqref{max_min_bnd_N_upper_proof} 
and this increasing property.
$\blacksquare$
\end{proof}

\begin{figure}[htb]
\centering
\subfigure[]{
\includegraphics[scale=0.5]{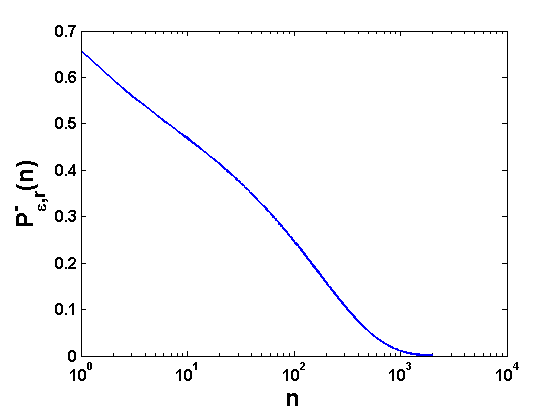}}
\subfigure[]{
\includegraphics[scale=0.5]{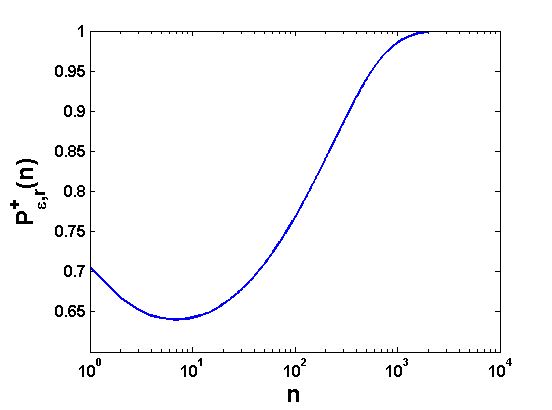}}
\caption{The curves of $P^{-}_{\veps,r}(n)$ and $P^{+}_{\veps,r}(n)$, 
defined in~\eqref{min_bnd_N_lower} and~\eqref{min_bnd_N_upper},
for $\veps = 0.1$ and $r=1$: 
(a) $P^{-}_{\veps,r}(n)$ decreases monotonically
for all $n \geq 1$; (b) $P^{+}_{\veps,r}(n)$ increases monotonically
only for $n \geq n_0$, where $n_0 > 1$:
according to  Theorem~\ref{main_trace_theorem_upper}, $n_0 = 100$ is safe,
and this value does not disagree with the plot.}
\label{fig:P_vs_n}
\end{figure}

\noindent \textbf{Remarks:} 
\begin{enumerate}[(i)]
	\item Part (iii) of Theorem~\ref{main_trace_theorem_upper} states that if $\delta$ is not small enough, 
then $n_{0}$ might not be a necessary and sufficient sample size for the special matrices mentioned there, 
i.e., matrices with $\lambda_{1} = \lambda_{2} = \cdots = \lambda_{r}$. 
This can be seen from Figure~\ref{fig:P_vs_n}(b): for $r=1, \veps = 0.1$, if $\delta = 0.33$, say,  
there is an integer $10 < n \leq 100$ such that~\eqref{prob_ineq_upper} holds,
so $n=101$ is no longer a necessary sample size (although it is still sufficient). 
	\item Simulations show that the sufficient sample size obtained using Theorems~\ref{main_trace_theorem_lower} 
	and~\ref{main_trace_theorem_upper}, amounts to bounds of the form 
	$\mathcal{O}\left(c(\veps,\delta) g(\delta)\right)$, where $g(\delta) < 1$ 
	is a decreasing function of $\delta$ and $c(\veps,\delta)$ is as defined in~\eqref{loose_bnd}.  
	As such, for larger values of $\delta$, i.e., when larger uncertainty is allowed, 
	one can obtain significantly smaller sample sizes than the one predicted by~\eqref{loose_bnd}; 
	see Figures~\ref{fig:compare_N1} and~\ref{fig:compare_N2}. 
	In other words, the difference between the above tighter conditions 
	and~\eqref{loose_bnd} is increasingly more prominent as $\delta$ gets larger.
	\item Note that the results in Theorems~\ref{main_trace_theorem_lower} and~\ref{main_trace_theorem_upper} are independent of the size of the matrix. In fact, the first items (i) in both theorems
	do not require any a priori knowledge about the matrix, other than it being SPSD. 
	In order to compute the necessary sample sizes, though, one is required to also know the rank of the matrix.
	\item The conditions in our theorems, despite their potentially ominous look,
 are actually simple to compute. Appendix~\ref{matlab} contains a short {\sc Matlab} code which calculates these 
necessary or sufficient  sample sizes to satisfy the probabilistic accuracy guarantees~\eqref{prob_ineq_lower_upper}, 
given a pair $(\veps,\delta)$ (and the matrix rank $r$ in case of necessary sample sizes).
%
%
%
%
%
%
This code was used for generating Figures~\ref{fig:compare_N1} and~\ref{fig:compare_N2}.
\end{enumerate}
\begin{figure}[htb]
\centering
\subfigure[]{
\includegraphics[scale=0.5]{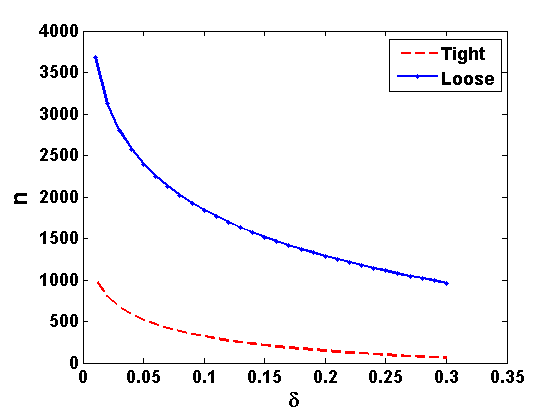}}
\subfigure[]{
\includegraphics[scale=0.5]{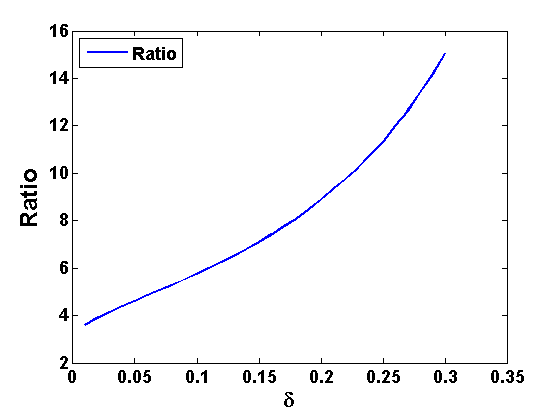}}
\caption{Comparing, as a function of $\delta$,
the sample size obtained from~\eqref{max_bnd_N_lower} and denoted by ``tight'', with
that of~\eqref{loose_bnd} and denoted by ``loose'', 
for $\veps = 0.1$ and $0.01 \leq \delta \leq 0.3$: 
(a) sufficient sample size, $n$, for~\eqref{prob_ineq_lower}, 
(b) ratio of sufficient sample size obtained from~\eqref{loose_bnd} over that of~\eqref{max_bnd_N_lower}.
When $\delta$ is relaxed, our new bound is tighter than the older one by an order of magnitude.}
\label{fig:compare_N1}
\end{figure}

Combining Theorems~\ref{main_trace_theorem_lower} and~\ref{main_trace_theorem_upper}, 
we can easily state conditions on the sample size $n$ for which the condition
\begin{equation}
\Pr \big( |tr_{n}(A) -tr(A)| \leq \veps~ tr(A) \big) \geq 1-\delta 
\label{prob_ineq}
\end{equation}
holds.
We have the following immediate corollary:

\begin{corollary}[\textbf{Necessary and sufficient condition for~\eqref{prob_ineq}}]
Given an SPSD matrix $A$ of rank $r$ and tolerances $(\veps,\delta)$ as above, the following hold:
\begin{enumerate}[(i)]
		\item \textbf{Sufficient condition:} if the inequality 
	\begin{equation}
			\Pr\big((1-\veps) \leq Q(n_{0}) \leq (1+\veps) \big) \geq 1-\delta
		\label{max_bnd_N}
	\end{equation}
	is satisfied for some $n_{0} > \veps^{-1}$, then~\eqref{prob_ineq} holds with $n = n_{0}$. 
	Furthermore, there is always an $n_{0} > \veps^{-2}$ such that \eqref{max_bnd_N}
	is satisfied and, for such $n_{0}$, it follows that~\eqref{prob_ineq} holds for all $n \geq n_{0}$.
	\item \textbf{Necessary condition:} if~\eqref{prob_ineq} 
	holds for some $n_{0} > \veps^{-1}$, then
	\begin{equation}
			\Pr\big((1-\veps) \leq Q(n r) \leq (1+\veps)\big) \geq 1- \delta,
		\label{min_bnd_N}
	\end{equation}
	with $n = n_{0}$. Furthermore, if $n_{0} > \veps^{-2}r^{-2}$, 
	then~\eqref{min_bnd_N} holds for all $n \geq n_{0}$.
	\item \textbf{Tightness:} if the $r$ positive eigenvalues of $A$ are all equal then
	there is a smallest $n_{0} > \veps^{-2} r^{-2}$ satisfying~\eqref{min_bnd_N}  
	such that for any $n \geq n_{0}$,~\eqref{prob_ineq} holds, 
	and for any $\veps^{-2} r^{-2} < n < n_{0}$,~\eqref{prob_ineq} does not hold. 
	If $\delta$ is small enough so that~\eqref{min_bnd_N} does not hold for any 
	$n \leq \veps^{-2} r^{-2}$, then $n_{0}$ is both necessary and sufficient for~\eqref{prob_ineq}. 
\end{enumerate}
\label{main_trace_corollary}
\end{corollary}
\noindent \textbf{Remark:} 
The necessary condition in Corollary~\ref{main_trace_corollary}(ii) is only valid for $n > \veps^{-1}$ 
(this is a consequence of the condition~\eqref{min_bnd_N} being tight, as shown in part (iii)). 
In~\cite{roas1}, 
an ``almost tight'' necessary condition is given that works for all $n\geq 1$.

\begin{figure}[htb]
\centering
\subfigure[]{
\includegraphics[scale=0.5]{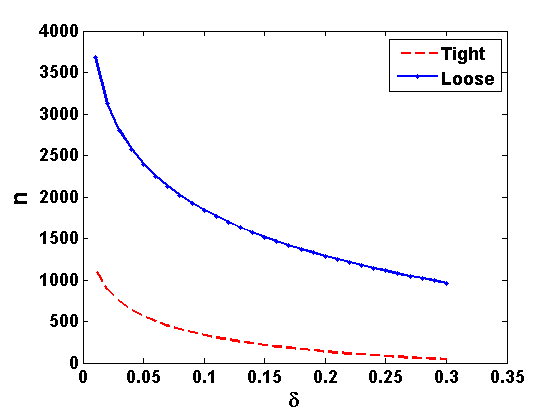}}
\subfigure[]{
\includegraphics[scale=0.5]{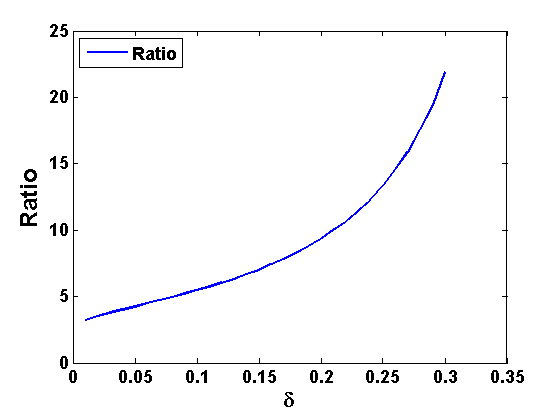}}
\caption{Comparing, as a function of $\delta$,
the sample size obtained from~\eqref{max_bnd_N_upper} and denoted by ``tight'', 
with that of~\eqref{loose_bnd} and denoted by ``loose'', 
for $\veps = 0.1$ and $0.01 \leq \delta \leq 0.3$: 
(a) sufficient sample size, $n$, for~\eqref{prob_ineq_upper}, 
(b) ratio of sufficient sample size obtained from~\eqref{loose_bnd} over that of~\eqref{max_bnd_N_upper}. 
When $\delta$ is relaxed, our new bound is tighter than the older one by an order of magnitude.}
\label{fig:compare_N2}
\end{figure}


\section{Randomized algorithms for solving large scale NLS problems}
\label{sec:alg_intro}

Consider the problem of decreasing the value of the original objective~\eqref{misfit}
to a desired level (e.g., satisfying a given tolerance) to recover the sought model, $\mm$. 
With the sensitivity matrices
\begin{eqnarray*}
J_i(\mm) = \frac{\partial \ff(\mm,\qq_{i})}{\partial \mm}, \quad i = 1, \ldots , s \label{2.9}
\end{eqnarray*}
we have the gradient
\begin{eqnarray*}
\grad \phi (\mm) = 2 \sum_{i=1}^{s} J_{i}^{T}(\mm)(\ff(\mm,\qq_{i}) - \dd_{i}).
\end{eqnarray*}

An iterative method such as modified Gauss-Newton (GN), L-BFGS, or nonlinear conjugate gradient
is typically designed to decrease the value of
the objective function using repeated calculations of the gradient. 
In the present article we follow~\cite{rodoas1} and employ variants of stabilized GN throughout, 
thus achieving
a context in which to focus our attention on the new aspects of this work.
In the $k^{th}$ iteration of such a method, having the current iterate $\mm_k$,
an update direction, $\delta \mm_k$, is calculated.
Then the iterate is updated as $\mm_{k+1} \leftarrow \mm_{k} + \alpha_k \delta \mm_k$, 
for some appropriate step length $\alpha_k$. 

What is special in our context here is that the update direction, $\delta \mm_k$, 
is calculated using the approximate misfit, $\widehat{\phi}(\mm_{k},n_{k})$,
defined as described in~\eqref{approx_phi} ($n_{k}$ is the sample size used for this approximation in the $k^{th}$ iteration). 
Thus, we need to check or assess whether the value of the original objective is also decreased using this new iterate.
The challenge is to do this as well as check for termination of the iteration process with
a minimal number of evaluations of the prohibitively expensive original misfit function $\phi$. 

In this section, we extend the algorithms introduced in~\cite{rodoas1,rodoas2} 
in the context of the more general NLS formulation~\eqref{misfit_iida} or~\eqref{misfit_inid},
assuming that their corresponding noise distributions hold,
although, as promised in Section~\ref{sec:intro}, 
we stick to the simpler notation~\eqref{misfit}, \eqref{approx_phi}.
  Variants of modified stochastic steps 
in the original algorithms are presented, 
and using Theorems~\ref{main_trace_theorem_lower} and~\ref{main_trace_theorem_upper}, 
the uncertainties in these steps are quantified. 
More specifically, in the main algorithm introduced in~\cite{rodoas1},
following a stabilized GN iteration on the approximated objective function 
using the approximated misfit, 
the iterate is  updated, and some (or all) of the following steps are performed: 
\begin{enumerate}[(i)]
\item {\em cross validation} -- approximate assessment of this iterate in terms of sufficient decrease in the
objective function using a control set of random combinations of measurements. More specifically, 
at the $k^{th}$ iteration with the new iterate $\mm_{k+1}$, we test whether the condition 
\begin{equation}
\widehat{\phi}(\mm_{k+1},n_{k}) \leq \kappa \widehat{\phi}(\mm_{k},n_{k})
\label{cross_valid}
\end{equation}
(cf.~\eqref{approx_phi}) holds for some $\kappa \leq 1$,
employing an independent set of weight vectors used in both approximations of $\phi$;
\item {\em uncertainty check}
-- upon success of cross validation, an inexpensive plausible termination
test is performed where,
given a tolerance $\rho$, we check for the condition
\begin{equation}
\widehat{\phi}(\mm_{k+1},n_{k}) \leq \rho
\label{uncert_check}
\end{equation}
using a fresh set of random weight vectors; and 
\item {\em stopping criterion} -- upon success of the uncertainty check, 
an additional independent and potentially more rigorous termination test 
against the given tolerance $\rho$ is performed (possibly using the original misfit function).
\end{enumerate}

The role of the cross validation step within an iteration is 
to assess 
whether the true objective function at the current iterate has (sufficiently) decreased compared to the previous one.
If this test fails, we deem that the current sample size is not sufficiently large to yield an update that decreases the original objective,
and the fitting step needs to be repeated using a larger sample size, see~\cite{doas3}. 
In~\cite{rodoas1}, this step was used heuristically, so 
the amount of uncertainty in such validation of the current iterate was not quantified. 
Consequently, 
there was no handle on the amount of false positives/negatives in such approximate evaluations 
(e.g., a sample size could be deemed too small while the stabilized GN iteration
has in fact produced an acceptable iterate). 
In addition, in~\cite{rodoas1} the sample size for the uncertainty check was heuristically chosen. 
So this step was also performed with no control over the amount of uncertainty. 

For the stopping criterion step in~\cite{rodoas1,doas3}, 
the objective function was accurately evaluated using all $s$ 
experiments, which is clearly a very expensive choice for an algorithm termination check.
This was a judicious decision made in order to be able to have a fairer comparison of the new and different methods 
proposed there. Replacement of this termination criterion by another independent heuristic ``uncertainty check''
is experimented with in~\cite{rodoas2}. 

In this section, we address the issues of quantifying the uncertainty in the validation, uncertainty check 
and stopping criterion steps within a nonlinear iteration. 
In what follows, we assume for simplicity that the iterations are performed on the 
objective~\eqref{misfit} using dynamic regularization 
(or iterative regularization \cite{hansen1998,doas1,doas3}) 
where the regularization is performed implicitly. Extension to the case~\eqref{objective} 
is straight forward. Throughout, we assume to be given a pair of positive and small 
probabilistic tolerance numbers, $(\veps,\delta)$.


\subsection{Cross validation step with quantified uncertainty}
\label{sec:cross_valid}
The condition~\eqref{cross_valid} is an independent, unbiased indicator of 
\begin{equation*}
\phi(\mm_{k+1}) \leq \kappa \phi(\mm_{k}),
\end{equation*}
which indicates sufficient decrease in the objective.
If~\eqref{cross_valid} is satisfied then the current sample size, $n_{k}$, 
is considered sufficiently large to capture the original misfit well enough to produce a valid 
iterate, and the algorithm continues using the same sample size. 
Otherwise, the sample size is deemed insufficient and is increased.
Using Theorems~\ref{main_trace_theorem_lower} and~\ref{main_trace_theorem_upper}, 
we can now remove the heuristic characteristic as to \textit{when} this sample size increase has been
performed hitherto,
and present two variants of~\eqref{cross_valid} where the uncertainties in the validation step are quantified. 

Assume 
we have a sample size $n_{c}$
such that
\begin{subequations}
\begin{eqnarray}
Pr\left( \widehat{\phi}(\mm_{k},n_{c}) \leq (1+\veps) \phi(\mm_{k}) \right) &\geq& 1-\delta ,\\
Pr\left( \widehat{\phi}(\mm_{k+1},n_{c}) \geq (1-\veps) \phi(\mm_{k+1}) \right) &\geq& 1-\delta.%
\end{eqnarray}%
\label{cross_valid_prob_hard}
\end{subequations}%
If in the procedure outlined above, after obtaining the updated iterate $\mm_{k+1}$, we verify that  
\begin{equation}
\widehat{\phi}(\mm_{k+1},n_{c}) \leq \kappa \left( \frac{1-\veps}{1+\veps} \right) \widehat{\phi}(\mm_{k},n_{c}),
\label{cross_valid_hard}
\end{equation}
then it follows from~\eqref{cross_valid_prob_hard} that $\phi(\mm_{k+1}) \leq \kappa \phi(\mm_{k})$ with a probability of, at least, $(1-\delta)^{2}$. In other words, success of~\eqref{cross_valid_hard} indicates that the updated iterate decreases the value of the original misfit~\eqref{misfit} with a probability of, at least, $(1-\delta)^{2}$.

Alternatively, suppose that we have
\begin{subequations}
\begin{eqnarray}
Pr\left( \widehat{\phi}(\mm_{k},n_{c}) \geq (1-\veps) \phi(\mm_{k}) \right) &\geq& 1-\delta ,\\
Pr\left( \widehat{\phi}(\mm_{k+1},n_{c}) \leq (1+\veps) \phi(\mm_{k+1}) \right) &\geq& 1-\delta.
\end{eqnarray}%
\label{cross_valid_prob_soft}%
\end{subequations}%
Now, if instead of~\eqref{cross_valid_hard} we check whether or not
\begin{equation}
\widehat{\phi}(\mm_{k+1},n_{c}) \leq \kappa \left( \frac{1+\veps}{1-\veps} \right) 
\widehat{\phi}(\mm_{k},n_{c}),
\label{cross_valid_soft}
\end{equation}
then it follows from~\eqref{cross_valid_prob_soft} that if the condition~\eqref{cross_valid_soft} 
is {\em not} satisfied, then $\phi(\mm_{k+1}) > \kappa \phi(\mm_{k})$ with a probability of, at least, 
$(1-\delta)^{2}$. In other words, failure of~\eqref{cross_valid_soft} indicates that the updated 
iterate results in an insufficient decrease 
in the original misfit~\eqref{misfit} with a probability of, at least, $(1-\delta)^{2}$.

We can replace~\eqref{cross_valid} with either of the conditions~\eqref{cross_valid_hard} 
or~\eqref{cross_valid_soft} and use the conditions~\eqref{max_bnd_N_lower} or~\eqref{max_bnd_N_upper} 
to calculate the cross validation sample size, $n_{c}$. If the relevant check~\eqref{cross_valid_hard} 
or~\eqref{cross_valid_soft} fails, we deem the sample size used in the fitting step, $n_{k}$, 
to be too small to produce an iterate which decreases the original misfit~\eqref{misfit}, 
and consequently consider increasing the sample size, $n_{k}$. Note that since 
$\frac{1-\veps}{1+\veps} < 1 < \frac{1+\veps}{1-\veps}$,
the condition~\eqref{cross_valid_hard} results in a more aggressive strategy for increasing 
the sample size used in the fitting step than the condition~\eqref{cross_valid_soft}. 
Figure~\ref{fig:pde_example_sample_size} in Section~\ref{sec:application} demonstrates this
within the context of an application.

\noindent \textbf{Remarks:} 
\begin{enumerate}[(i)]
\item Larger values of $\veps$ result in more aggressive (or relaxed) descent requirement by the condition~\eqref{cross_valid_hard} (or~\eqref{cross_valid_soft}).  
\item 
As the iterations progress and we get closer to the solution, 
the decrease in the original objective could be less than what is imposed by~\eqref{cross_valid_hard}. 
As a result, if $\veps$ is too large, we might never successfully pass the cross validation test. 
One useful strategy to alleviate this is to start with a larger $\veps$, 
decreasing it as we get closer to the solution. 
A similar strategy can be adopted for the case when the condition~\eqref{cross_valid_soft} is used as 
a cross validation: as the iterations get closer to the solution, one can make the condition~\eqref{cross_valid_soft} less relaxed by decreasing $\veps$. 
\end{enumerate}



\subsection{Uncertainty check with quantified uncertainty and efficient stopping criterion}
\label{sec:uncert_check}
The usual test for terminating the iterative process is to check whether 
\begin{equation}
\phi(\mm_{k+1}) \leq \rho ,
\label{stop_crit}
\end{equation}
for a given tolerance $\rho$.
However, this can be very expensive in our current context;
%
see Section~\ref{sec:pde_manys} and Tables~\ref{table01_2level} and~\ref{table01_3level} for examples of a scenario where one misfit evaluation using the entire data set can be as expensive as the entire cost of an efficient, complete algorithm. 
In addition, if 
the exact value of the tolerance $\rho$ is not known
(which is usually the case in practice), 
one should be able to reflect such uncertainty in the stopping criterion and perform a 
softer version of~\eqref{stop_crit}. 
Hence, it could be useful to have an algorithm which allows one to adjust the cost and accuracy 
of such an evaluation in a quantifiable way, and find the balance that is suitable to particular objectives 
and computational resources. 

Regardless of the issues of cost and accuracy, this evaluation should be carried out as 
rarely as possible and only when deemed timely. In~\cite{rodoas1}, we addressed this by 
employing an ``uncertainty check''~\eqref{uncert_check} as described earlier in this section, heuristically.
%
Using Theorems~\ref{main_trace_theorem_lower} and~\ref{main_trace_theorem_upper}, 
we now devise variants of~\eqref{uncert_check} with quantifiable uncertainty. 
Subsequently, again using Theorems~\ref{main_trace_theorem_lower} and~\ref{main_trace_theorem_upper}, we present a much cheaper stopping criterion than~\eqref{stop_crit} which, at the same time, reflects our uncertainty in the given tolerance. 

Assume that we have
a sample size $n_{u}$
such that
\begin{equation}
Pr\left( \widehat{\phi}(\mm_{k+1},n_{u}) \geq (1-\veps) \phi(\mm_{k+1}) \right) \geq 1-\delta.
\label{uncert_check_prob_hard}
\end{equation}
If the updated iterate, $\mm_{k+1}$, successfully passes the cross validation test, then we check for  
\begin{equation}
\widehat{\phi}(\mm_{k+1},n_{u}) \leq (1-\veps) \rho .
\label{uncert_check_hard}
\end{equation}
If this holds too then it follows from~\eqref{uncert_check_prob_hard} that $\phi(\mm_{k+1}) \leq \rho$ with a probability of, at least,  $(1-\delta)$. In other words, success of~\eqref{uncert_check_hard} indicates that the misfit is likely to be below the tolerance with a probability of, at least, $(1-\delta)$.

Alternatively, suppose that
\begin{eqnarray}
Pr\left( \widehat{\phi}(\mm_{k+1},n_{u}) \leq (1+\veps) \phi(\mm_{k+1}) \right) \geq 1-\delta ,
\label{uncert_check_prob_soft}
\end{eqnarray}
and instead of~\eqref{uncert_check_hard} we check for
\begin{equation}
\widehat{\phi}(\mm_{k+1},n_{u}) \leq (1+\veps) \rho .
\label{uncert_check_soft}
\end{equation}
then it follows from~\eqref{uncert_check_prob_soft} that if the condition~\eqref{uncert_check_soft} is {\em not} satisfied, then $\phi(\mm_{k+1}) > \rho$ with a probability of, at least,  $(1-\delta)$. In other words, failure of~\eqref{uncert_check_soft} indicates that using the updated iterate, the misfit is likely to be  still above the desired tolerance with a probability of, at least,  $(1-\delta)$.

We can replace~\eqref{uncert_check} with the condition~\eqref{uncert_check_hard} 
(or~\eqref{uncert_check_soft}) and use the condition~\eqref{max_bnd_N_lower} (or~\eqref{max_bnd_N_upper}) 
to calculate the uncertainty check sample size, $n_{u}$. If the test~\eqref{uncert_check_hard} 
(or~\eqref{uncert_check_soft}) fails then we skip the stopping criterion check and continue iterating. 
Note that since $(1-\veps) < 1 < (1+\veps)$,
the condition~\eqref{uncert_check_hard} results in fewer false positives than the 
condition~\eqref{uncert_check_soft}. On the other hand, the condition~\eqref{uncert_check_soft} is 
expected to results in fewer false negatives than the condition~\eqref{uncert_check_hard}. 
The choice of either alternative is dependent on one's requirements, resources and the application on hand.

The stopping criterion step can be performed in the same way as the uncertainty check but 
potentially with higher certainty in the outcome. 
In other words, for the stopping criterion we can choose a smaller $\delta$, resulting in a larger sample size 
$n_{t}$ satisfying $n_t > n_{u}$, and check for satisfaction of either
\begin{subequations}
\begin{equation}
\widehat{\phi}(\mm_{k+1},n_{t}) \leq (1-\veps) \rho ,\label{stop_crit_hard}
\end{equation}
or
\begin{equation}
\widehat{\phi}(\mm_{k+1},n_{t}) \leq (1+\veps) \rho . \label{stop_crit_soft}
\end{equation}
\label{stop_crits}
\end{subequations}
Clearly the condition~\eqref{stop_crit_soft} is a softer than~\eqref{stop_crit_hard}: a 
successful~\eqref{stop_crit_soft} is only necessary and not sufficient for concluding that~\eqref{stop_crit} 
holds with the prescribed probability. 

In practice, when the value of the stopping criterion threshold, $\rho$, is not {\em exactly} known 
(it is often crudely estimated using the measurements), 
one can reflect such uncertainty in $\rho$ by choosing an appropriately large $\delta$. 
Smaller values of $\delta$ reflect a higher certainty in $\rho$ and a more rigid stopping criterion.

\noindent \textbf{Remarks:} 
\begin{enumerate}[(i)]
\item If $\veps$ is large 
then using~\eqref{stop_crit_hard}, one might run the risk of over-fitting. 
Similarly, using~\eqref{stop_crit_soft} with large $\veps$, there is a risk of under-fitting. 
Thus, appropriate values of $\veps$ need to be considered in accordance with the application 
and one's computational resources and experience.
\item The same issues regarding large $\veps$ arise when employing the uncertainty check 
condition~\eqref{uncert_check_hard} (or ~\eqref{uncert_check_soft}): 
large $\veps$ might increase the frequency of false negatives (or positives).
\end{enumerate}

\subsection{Algorithm}
\label{sec:alg}
We now present an efficient, stochastic, iterative algorithm for approximately solving NLS formulations of~\eqref{misfit} or~\eqref{objective}. 
By performing cross validation, uncertainty check and stopping criterion as descried in Section~\ref{sec:cross_valid} 
and Section~\ref{sec:uncert_check}, we can devise 8 variants of Algorithm~\ref{alg1} 
below. Depending on the application, the variant of choice can be selected appropriately. 
More specifically, cross validation, uncertainty check and stopping criterion can, 
respectively, be chosen to be one of the following combinations (referring to their equation numbers):

\begin{table*}[htb]
\begin{center}
\addtolength{\tabcolsep}{-1.5pt}
\begin{tabular}{|c|c|c|c|}
\hline 
(i) (\ref{cross_valid_hard} - \ref{uncert_check_hard} - \ref{stop_crit_hard}) &
(ii) (\ref{cross_valid_hard} - \ref{uncert_check_hard} - \ref{stop_crit_soft}) & 
(iii) (\ref{cross_valid_hard} - \ref{uncert_check_soft} - \ref{stop_crit_hard}) & 
(iv) (\ref{cross_valid_hard} - \ref{uncert_check_soft} - \ref{stop_crit_soft}) \\ \hline
(v) (\ref{cross_valid_soft} - \ref{uncert_check_hard} - \ref{stop_crit_hard}) &
(vi) (\ref{cross_valid_soft} - \ref{uncert_check_hard} - \ref{stop_crit_soft}) &
(vii) (\ref{cross_valid_soft} - \ref{uncert_check_soft} - \ref{stop_crit_hard}) &
(viii) (\ref{cross_valid_soft} - \ref{uncert_check_soft} - \ref{stop_crit_soft}) \\ \hline
\end{tabular}
\end{center}
\end{table*}

\noindent \textbf{Remark:} 
\begin{enumerate}[(i)]
\item The sample size, $n_k$, used in the fitting step of Algorithm~\ref{alg1} 
could in principle be determined by
Corollary~\ref{main_trace_corollary}, using a pair of tolerances $(\veps_{f},\delta_{f})$.
If cross validation~\eqref{cross_valid_hard} (or~\eqref{cross_valid_soft}) fails, the tolerance pair $(\veps_{f},\delta_{f})$ is reduced to obtain, in the next iteration, a larger fitting sample size, $n_{k+1}$.
This would give a sample size which yields 
 a quantifiable approximation with a desired relative accuracy. 
 However, in the presence of all the added safety steps described in this section, 
 we have found in practice that Algorithm~\ref{alg1} is capable of producing a satisfying recovery, 
 even with a significantly smaller $n_{k}$ than the one predicted by Corollary~\ref{main_trace_corollary}. 
 Thus, the ``\textit{how}'' of the fitting sample size increase is left to heuristic 
 (as opposed to its ``\textit{when}'', which is quantified as described in Section~\ref{sec:cross_valid}). 
 \item In the algorithm below, we only consider fixed values (i.e., independent of $k$) for $\veps$ 
and $\delta$. One can easily modify Algorithm~\ref{alg1} 
 to incorporate non-stationary values which adapt to the iteration process, 
as mentioned in the closing remark of Section~\ref{sec:cross_valid}.
\end{enumerate}
In Algorithm~\ref{alg1}, when we draw vectors $\ww_i$ for some purpose, we always draw them independently
from the standard normal distribution.


\section{A practical application}
\label{sec:application}
In this section, we demonstrate the efficacy of Algorithm~\ref{alg1} by applying it to an 
important class of problems that arise often in practice: 
large scale partial differential equation (PDE) inverse problems with many measurements. 
We show below the capability of our method by applying it to such examples 
in the context of the DC resistivity/EIT problem, as in~\cite{doas3,rodoas1,rodoas2}.

\subsection{PDE inverse problems with many measurements}
\label{sec:pde_manys}
The context considered here is one where each evaluation of $\ff_i(\mm)$ in~\eqref{forward_op}
is computationally expensive.
The evaluation of the misfit function $\phi (\mm)$ is especially costly when many
experiments, involving different combinations of sources and receivers, 
are employed in order to obtain reconstructions of acceptable quality. 
The sought model $\mm$ is a discretization of the function $m(\xx)$ as described in
Section~\ref{sec:intro},
and
\begin{subequations}
\begin{eqnarray}
\ff_{i}(\mm) = P_i\uu_i = P_i L(\mm)^{-1} \qq_i . \label{1.5b}
\end{eqnarray}
Here we write the PDE system
in discretized form as
\begin{eqnarray}
L(\mm) \uu_i = \qq_i, \quad i = 1, \ldots , s,
\label{1.5a}
\end{eqnarray}
where $\uu_i \in \R^{l_q}$ is the $i$th field, $\qq_i \in \R^{l_q}$ is the $i$th source, 
and $L$ is a square matrix discretizing the PDEs plus appropriate side conditions.
Furthermore, the given projection matrices $P_i$ are such that $\ff_i(\mm)$
predicts the $i$th data set.
Note that the notation \eqref{1.5a} reflects an assumption of linearity in $\uu$ but not in $\mm$~\cite{rodoas1}. 
\label{1.5}
\end{subequations}

\begin{algorithm}
\caption{Solve NLS formulation of~\eqref{misfit} (or~\eqref{objective}) using uncertainty check, 
cross validation and cheap stopping criterion}
\begin{algorithmic}
\STATE \textbf{Given:} sources $\qq_{i}\;, i=1,\ldots,s$, measurements 
$\dd_{i}\;, i=1,\ldots,s$,  stopping criterion level $\rho$, objective function sufficient decrease factor $\kappa \leq 1$, pairs of small numbers $(\veps_{c},\delta_{c})$, $(\veps_{u},\delta_{u})$, 
$(\veps_{t},\delta_{t})$, and initial guess $\mm_{0}$.
\STATE \textbf{Initialize}: 
\STATE - $\mm = \mm_{0} \; , \; n_{0} = 1$
\STATE - Calculate the cross validation sample size, $n_{c}$, as described in Section~\ref{sec:cross_valid} with $(\veps_{c},\delta_{c})$.
\STATE - Calculate the sample sizes for uncertainty check, $n_{u}$, and stopping criterion, $n_{t}$, as described in Section~\ref{sec:uncert_check} with $(\veps_{u},\delta_{u})$ and $(\veps_{t},\delta_{t})$, respectively.
\FOR {$k = 0,1,2, \cdots$ until termination} 
\STATE \textbf{Fitting}: 
\STATE - Draw $\ww_{i}\;, i=1,\ldots,n_{k}$. 
\STATE - Approximate the misfit term and potentially its gradient in~\eqref{misfit} or~\eqref{objective} 
using~\eqref{approx_phi} with the above weights and $n = n_{k}$.
\STATE - Find an update for the objective function using the approximated misfit~\eqref{approx_phi}.
\STATE \textbf{Cross Validation}: 
\STATE - Draw $\ww_{i}\;, i=1,\ldots,n_{c}$. 
\IF {\eqref{cross_valid_hard} (or~\eqref{cross_valid_soft}) holds}
\STATE \textbf{Uncertainty Check}: 
\STATE - Draw $\ww_{i}\;, i=1,\ldots,n_{u}$. 
\IF {\eqref{uncert_check_hard} (or~\eqref{uncert_check_soft}) holds}
\STATE \textbf{Stopping Criterion}: 
\STATE - Draw $\ww_{i}\;, i=1,\ldots,n_{t}$. 
\IF {\eqref{stop_crit_hard} (or~\eqref{stop_crit_soft}) holds}
\STATE - Terminate
\ENDIF
\ENDIF
\STATE - Set $n_{k+1} = n_k$.
\ELSE
\STATE - \textbf{Sample Size Increase}: for example, set $n_{k+1} = \min(2 n_{k},s)$.
\ENDIF
\ENDFOR
\end{algorithmic}
\label{alg1}
\end{algorithm}


If the locations where data are measured do not change from one experiment to another, 
i.e., $P = P_{i}, \forall i$, then we get
\begin{eqnarray}
\ff(\mm,\qq_{i}) = P L(\mm)^{-1} \qq_i \label{pde_forward},
\end{eqnarray}
and the linearity assumption of $\ff(\mm,\qq)$ in $\qq$ is satisfied. 
Thus, we can use Algorithm~\ref{alg1} to efficiently recover $\mm$ and be quantifiably 
confident in the recovered model. 
If the $P_{i}$'s are different across experiments, there are methods to extend the existing data set 
to one where all sources share the same receivers, see~\cite{rodoas2,hach12}. 
Using these methods (when they apply!), one can effectively transform the problem~\eqref{1.5b} 
to~\eqref{pde_forward}, for which Algorithm~\ref{alg1} can be employed.

There are several problems of practical interest in the form~\eqref{misfit},~\eqref{1.5},
where the use of many experiments, resulting in a large number $s$,
is crucial for obtaining credible reconstructions in practical situations.
These include electromagnetic data inversion in mining exploration 
(e.g.,~\cite{na,dmr,haasol,olhash}), 
seismic data inversion in oil exploration (e.g.,~\cite{fichtner,hel,rnkkda}), diffuse optical tomography (DOT) (e.g.,~\cite{arridge,boas}), quantitative photo-acoustic tomography (QPAT) (e.g.,~\cite{gaooscher,yuan}), 
direct current (DC) resistivity (e.g.,~\cite{smvoz,pihakn,haheas,HaberChungHermann2010,doas3}),  
and electrical impedance tomography (EIT) (e.g.,~\cite{bbp,cin,doasha}).

Our examples are performed in the context of solving the DC resistivity problem.
The PDE has the form
\begin{subequations}
\begin{eqnarray}
\div (\mu(\xx) \grad u) = q(\xx), \quad \xx \in  \Omega ,
\label{2.1a}
\end{eqnarray}
where $\Omega \subset \R^d$, $d = 2$ or $3$, and
$\mu(\xx)$ is a conductivity function which may be rough\footnote{In theory, the conductivity function is defined so that $\mu \in L_{\infty}(\Omega)$, and hence it can be very rough.} (e.g., discontinuous).
However, the PDE is coercive: there is
a constant $\mu_0 > 0$ such that $\mu(\xx) \geq \mu_{0} , \; \forall \xx \in \Omega$.
It is possible to inject some a priori information on $\mu$, when such is available, via a parametrization
of $\mu(\xx)$ in terms of $m(\xx)$ using an appropriate transfer function $\psi$ as $\mu(\xx) = \psi(m(\xx))$. 
For example, $\psi$ can be chosen so as to ensure that the conductivity stays positive and bounded away from $0$, as well as to incorporate bounds, which are often known in practice, on the sought conductivity function. 
Some possible choices of function $\psi$ are described in~\cite[Appendix A]{rodoas1}. Here we take $\Omega$ to be the unit square in 2D, and assume 
the homogeneous Neumann boundary conditions
\begin{eqnarray}
\frac{\partial u}{\partial n} = 0, \quad \xx \in \partial\Omega. \label{2.1b}
\end{eqnarray}
\label{2.1}
\end{subequations}

The inverse problem is then to recover $m$ in $\Omega$ from sets of measurements of $u$
on the domain's boundary for different sources $q$. 
Details of the numerical methods employed here, both for defining the predicted data $\ff$
and for solving the inverse problem in appropriately transformed variables, can be found
in~\cite[Appendix A]{rodoas1}. 

\subsection{Numerical experiments}
\label{sec:simulations}

Below we consider two examples, each having a piecewise constant ``exact solution'',
or ``true model'', used to synthesize data:
\begin{enumerate}[(E.1)]
	\item in our simpler model a target object with conductivity 
$\mu_{t} = 1$ has been placed in a background medium 
with conductivity $\mu_{b} = 0.1$ (see Figure~\ref{fig:pde_example_recovery_Vanilla_2_level}(a)); and \label{numer_setup_a}
\item in a slightly more complex setting a conductive object with conductivity $\mu_{c} = 0.01$,
as well as a resistive one with conductivity
$\mu_{r} = 1$, have been placed in a background medium 
with conductivity $\mu_{b} = 0.1$ (see Figure~\ref{fig:pde_example_recovery_Vanilla_3_level}(a)). Note that the recovery of the model in Example~\hyperref[numer_setup_b]{(E.2)} is more challenging than Example~\hyperref[numer_setup_a]{(E.1)} since here the dynamic range of the conductivity is much larger. \label{numer_setup_b}
\end{enumerate}
Details of the numerical setup for the following examples are given in Appendix~\ref{numerical_setup}.

\subsubsection{Example~\hyperref[numer_setup_a]{(E.1)}}

We carry out the 8 variants of Algorithm~\ref{alg1} for the parameter values
$(\veps_{c},\delta_{c}) = (0.05,0.3)$, $(\veps_{u},\delta_{u}) = (0.1,0.3)$, $(\veps_{t},\delta_{t}) = (0.1,0.1)$, 
and $\kappa = 1$. 
The resulting total count of PDE solves, which is 
the main computational cost of the iterative solution of such inverse problems, is reported 
in Tables~\ref{table01_2level} and~\ref{table01_3level}. As a point of reference, we
also include the total PDE count using the ``plain vanilla'' stabilized Gauss-Newton
method which employs the entire set of $s$ experiments at every iteration and misfit estimation task. 
The recovered conductivities are displayed
in Figures~\ref{fig:pde_example_recovery_2_level} and~\ref{fig:pde_example_recovery_3_level}, 
demonstrating that employing Algorithm~\ref{alg1} 
can drastically reduce the total work while obtaining equally acceptable reconstructions.

\begin{table}[!ht]
\begin{center}
\begin{tabular}{|ccccccccc|}
\hline 
Vanilla & (i) & (ii) & (iii) & (iv) & (v) & (vi) & (vii) & (viii)
\\ \hline 
436,590 & 4,058 & 4,028 & 3,764 & 3,282 & 4,597 & 3,850 & 3,734 & 3,321\\ \hline
\end{tabular}
\end{center}
\caption{Example~\hyperref[numer_setup_a]{(E.1)}. Work in terms of number of PDE solves for all variants of Algorithm~\ref{alg1},
described in Section~\ref{sec:alg} and indicated here by (i)--(viii). 
The ``vanilla'' count is also given, as a reference. \label{table01_2level}}
\end{table}

\begin{figure}[htb]
\centering
\subfigure[]{\includegraphics[scale=0.15]{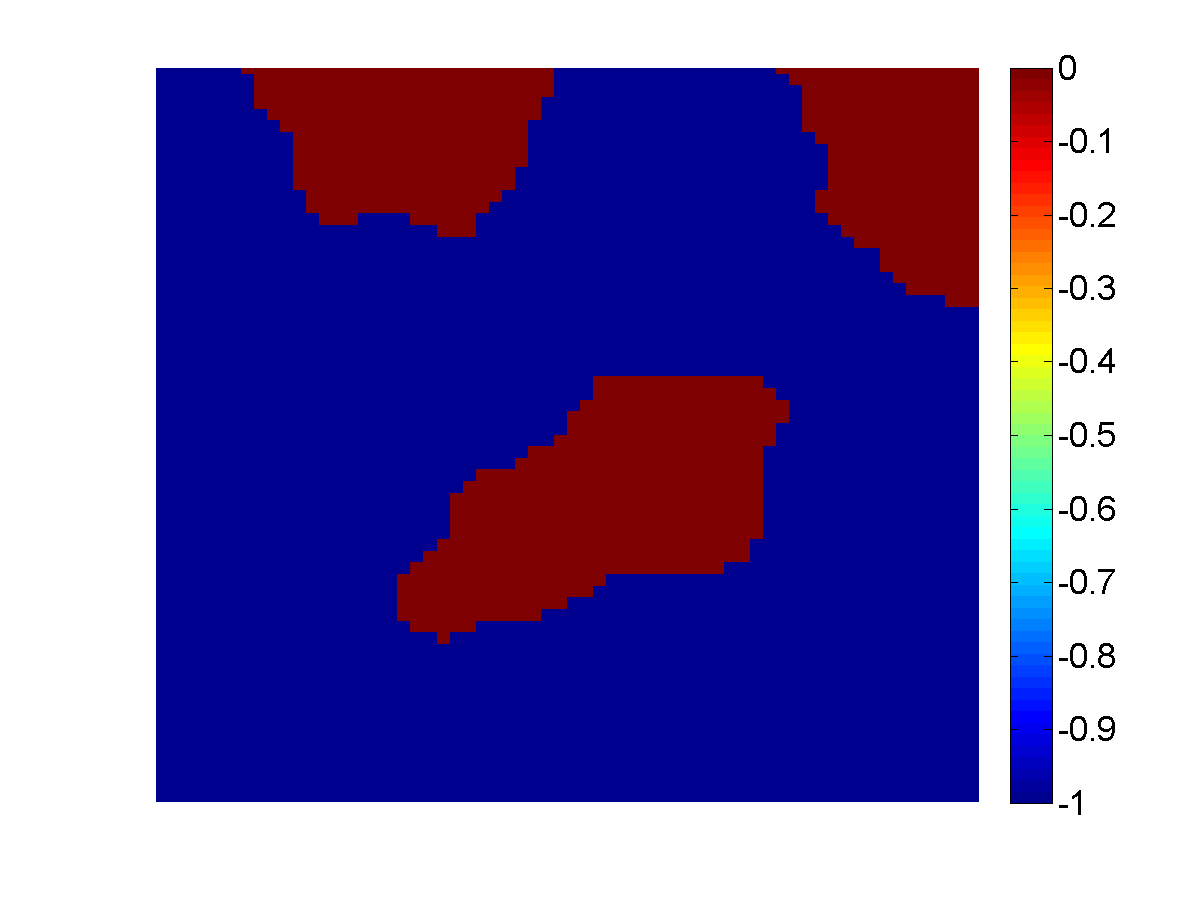}}
\subfigure[]{\includegraphics[scale=0.15]{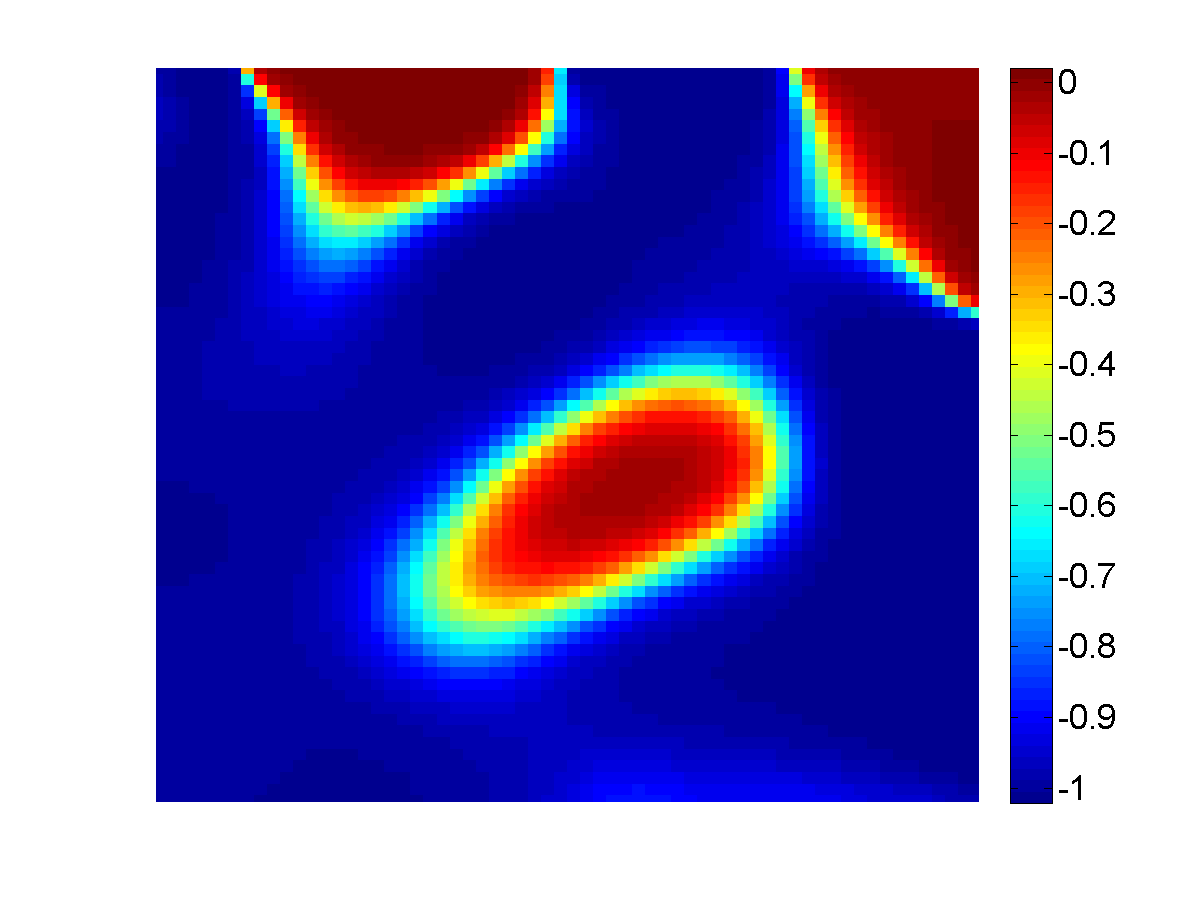}}
\subfigure[]{\includegraphics[scale=0.15]{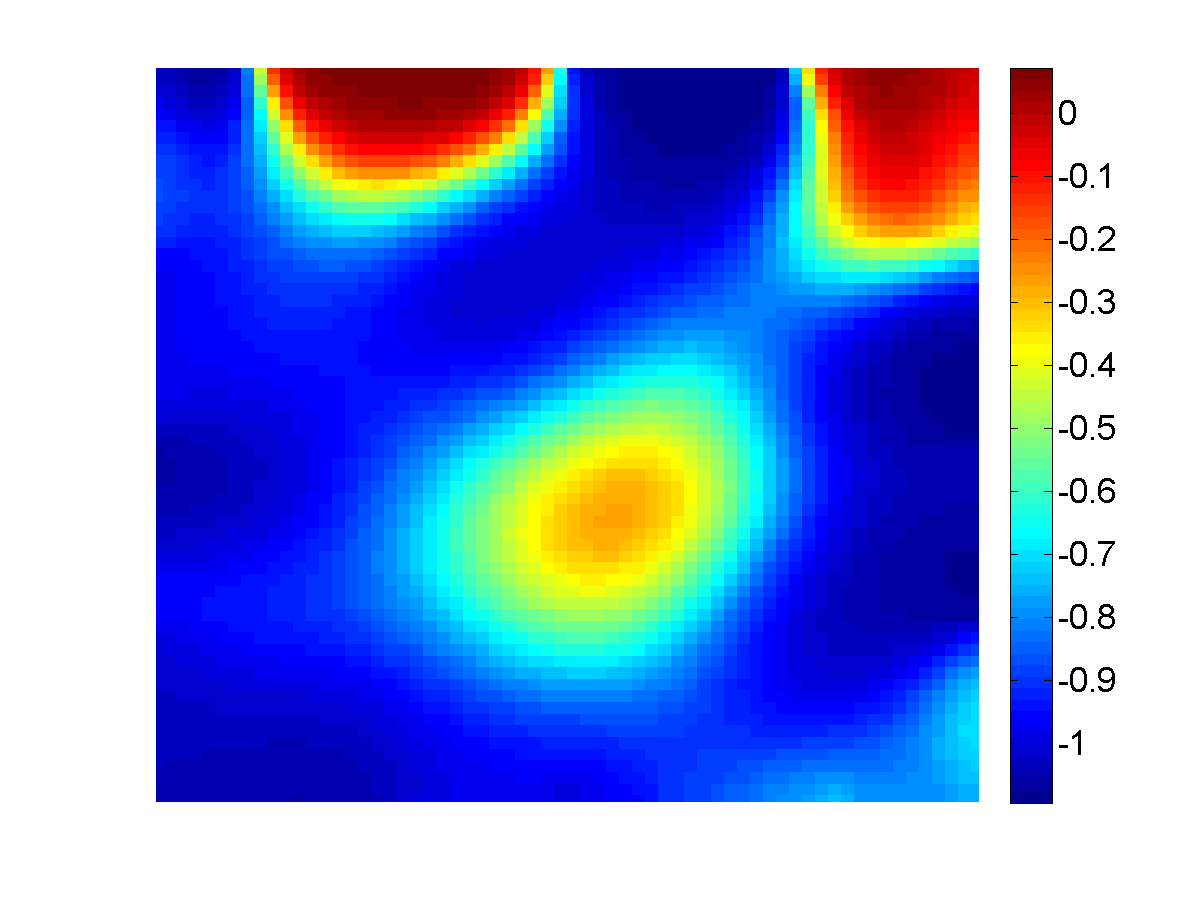}}
\caption{Example~\hyperref[numer_setup_a]{(E.1)}. Plots of log-conductivity: (a) True model; (b) Vanilla recovery with $s = 3,969$; (c) Vanilla recovery with $s=49$. The vanilla recovery using only $49$ measurement sets is clearly inferior, 
showing that a large number of measurement sets can be crucial for better reconstructions.}
\label{fig:pde_example_recovery_Vanilla_2_level}		
\renewcommand{\thesubfigure}{(\roman{subfigure})}
\subfigure[]{
\includegraphics[scale=0.15]{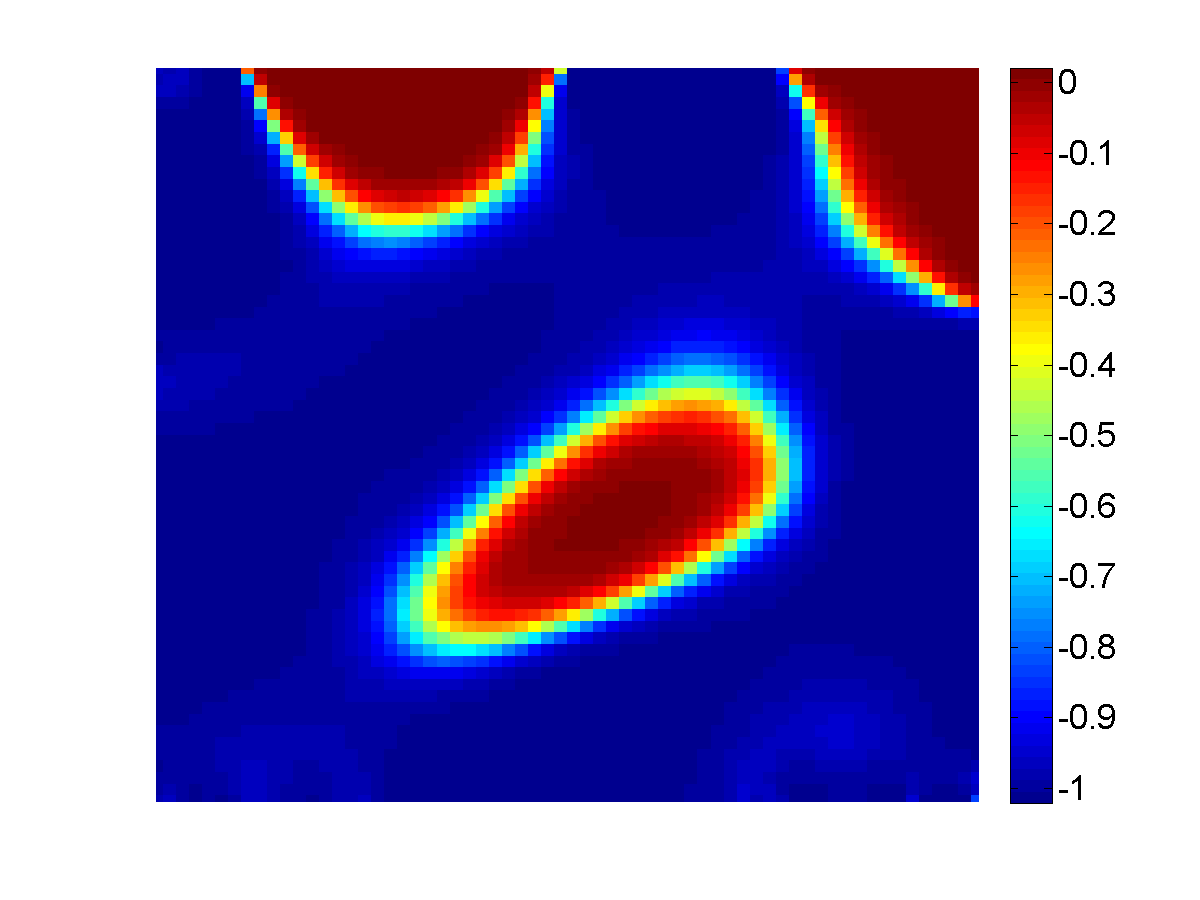}}
\subfigure[]{
\includegraphics[scale=0.15]{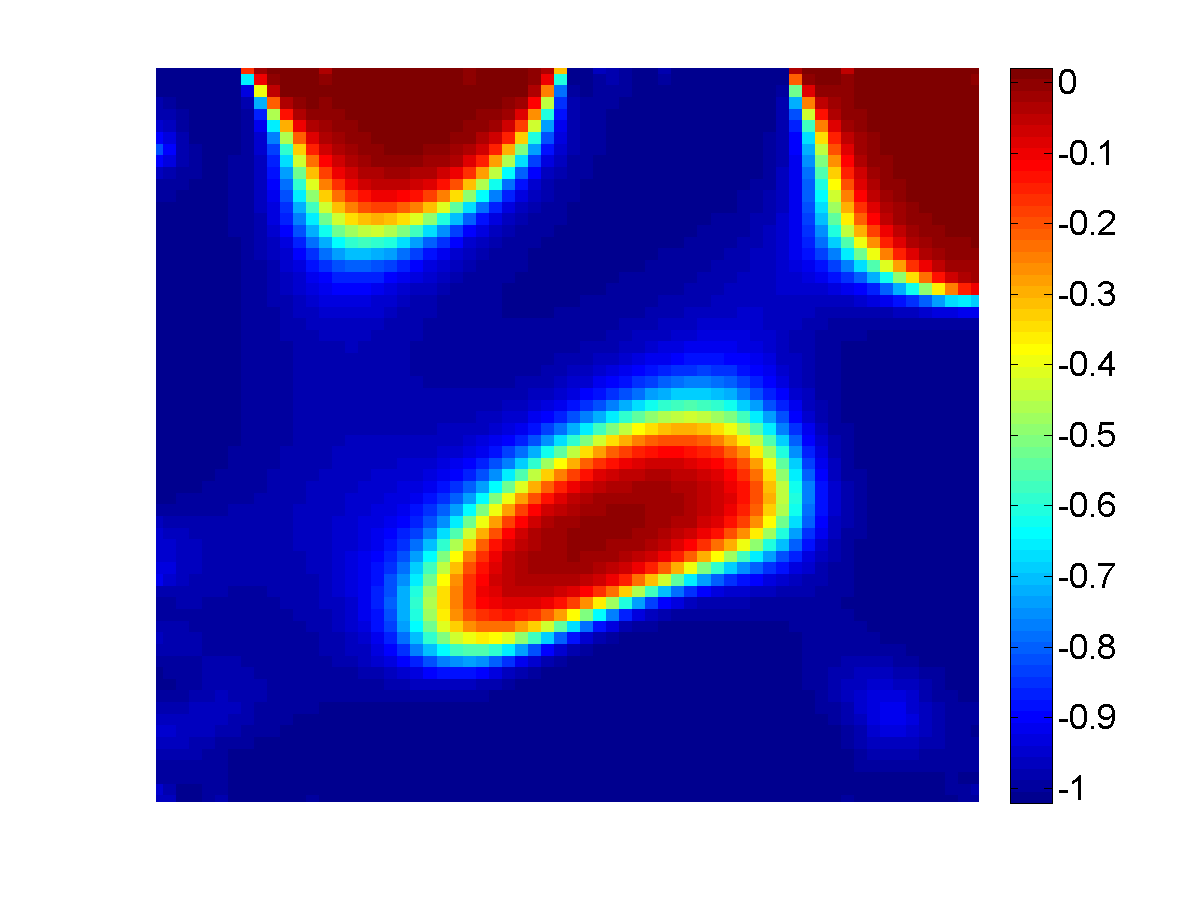}}
\subfigure[]{
\includegraphics[scale=0.15]{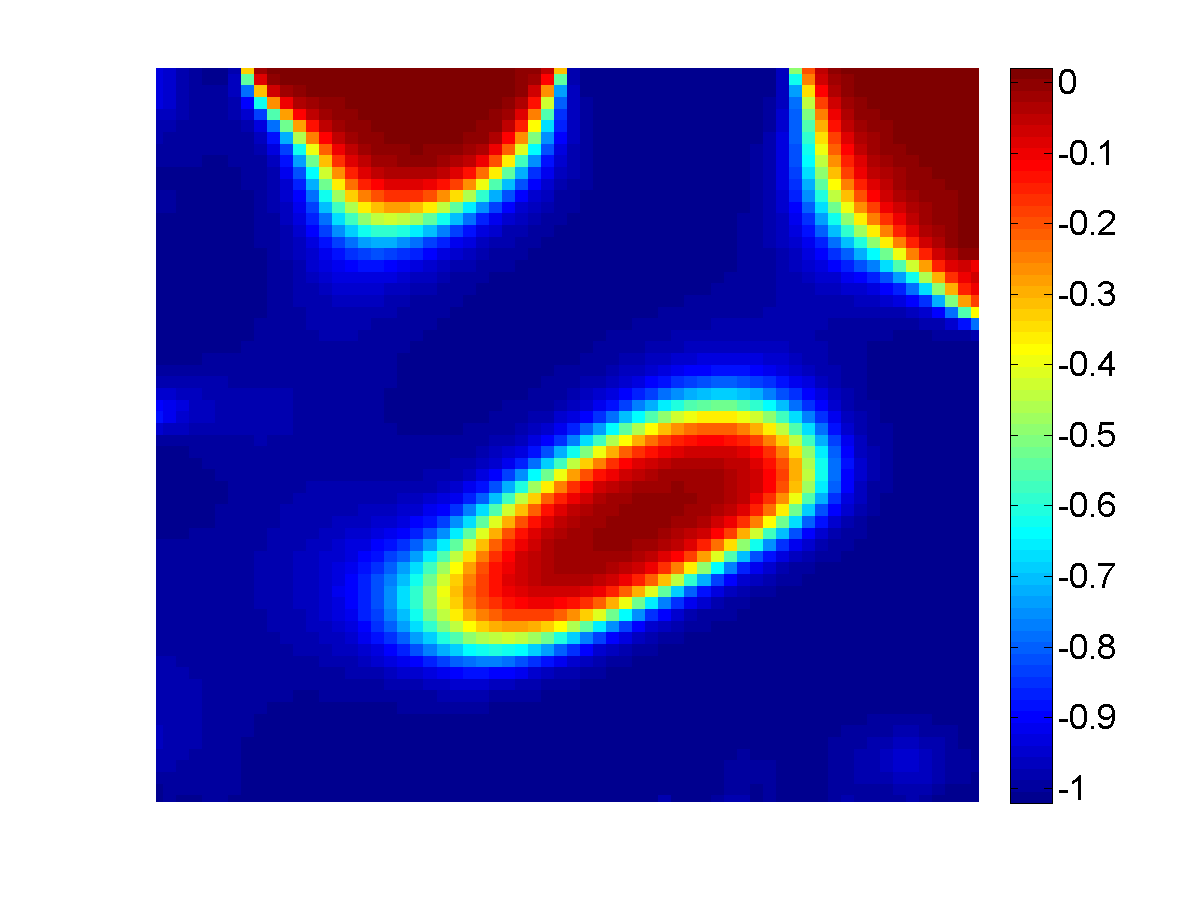}}
\subfigure[]{
\includegraphics[scale=0.15]{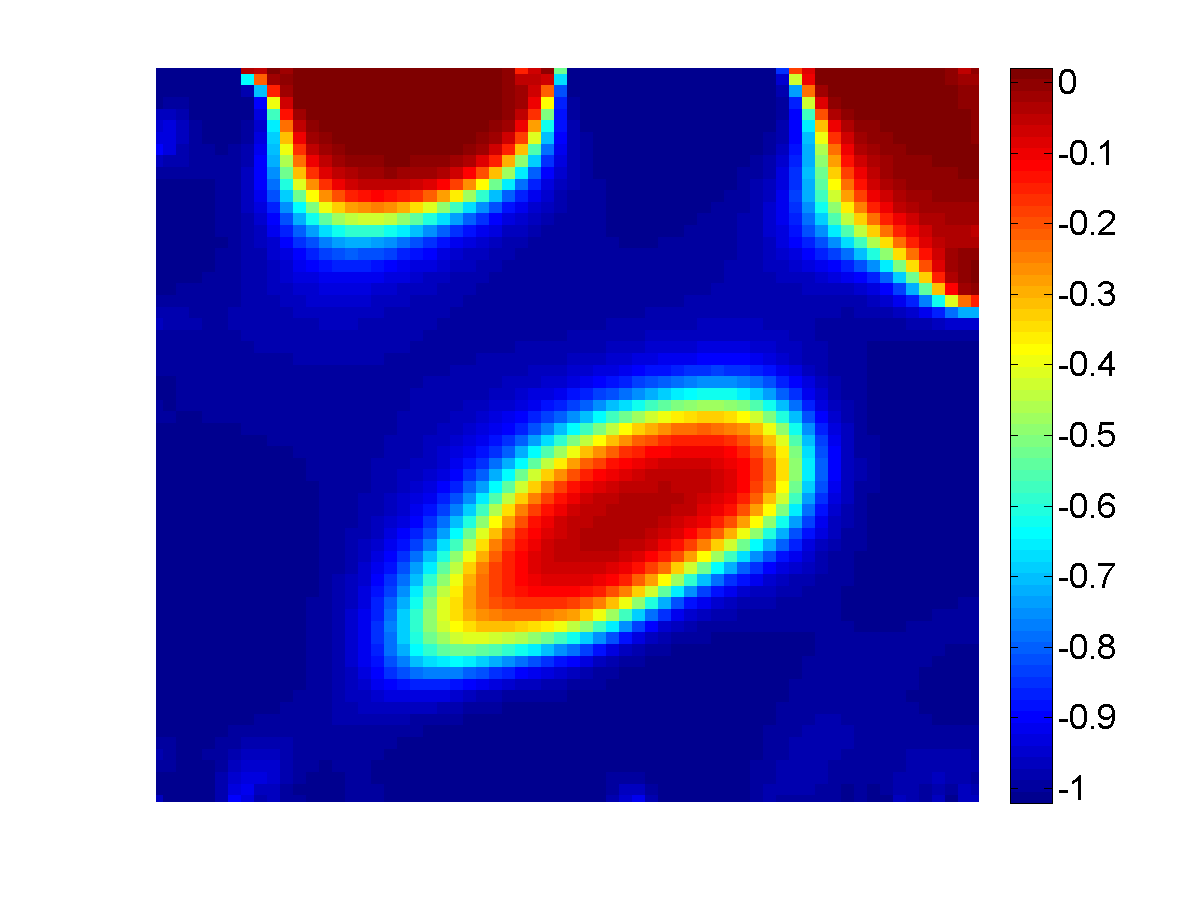}}
\subfigure[]{
\includegraphics[scale=0.15]{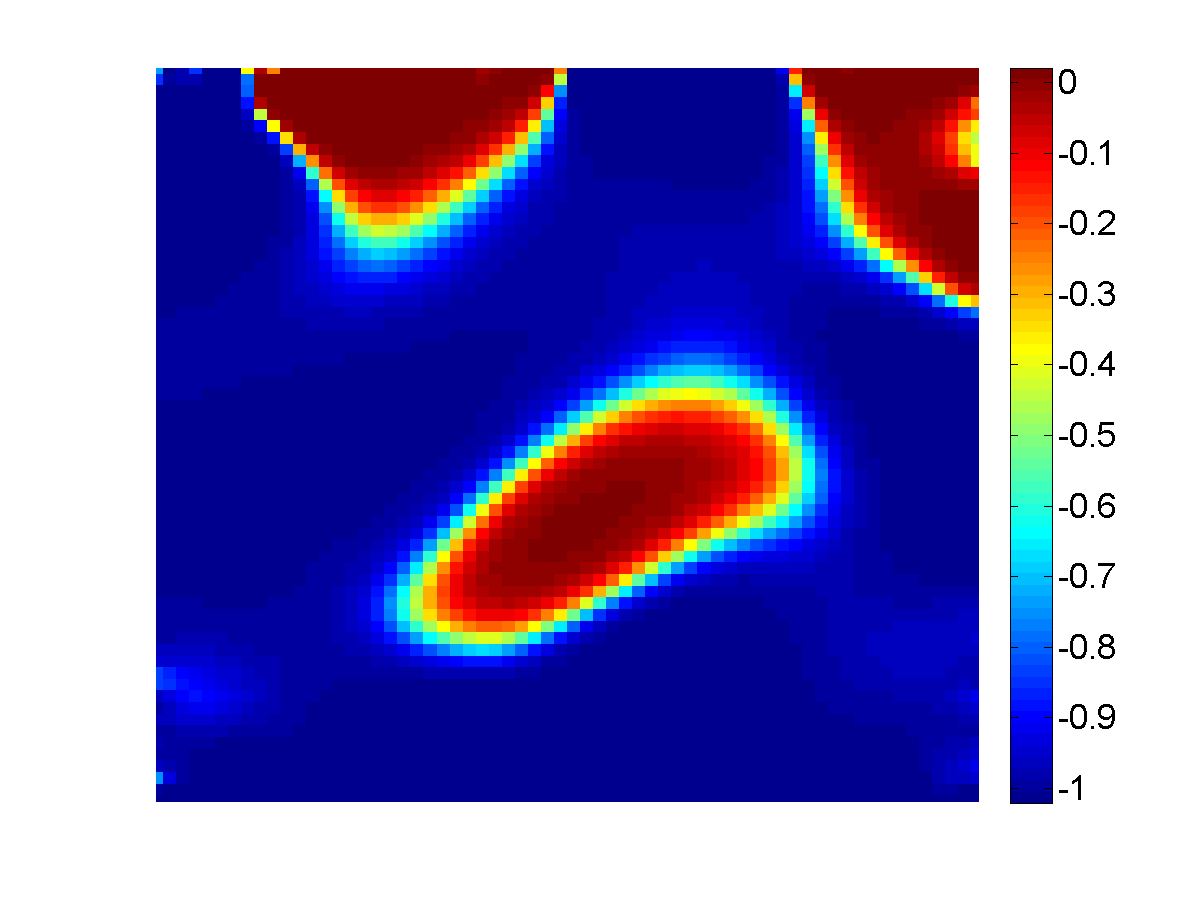}}
\subfigure[]{
\includegraphics[scale=0.15]{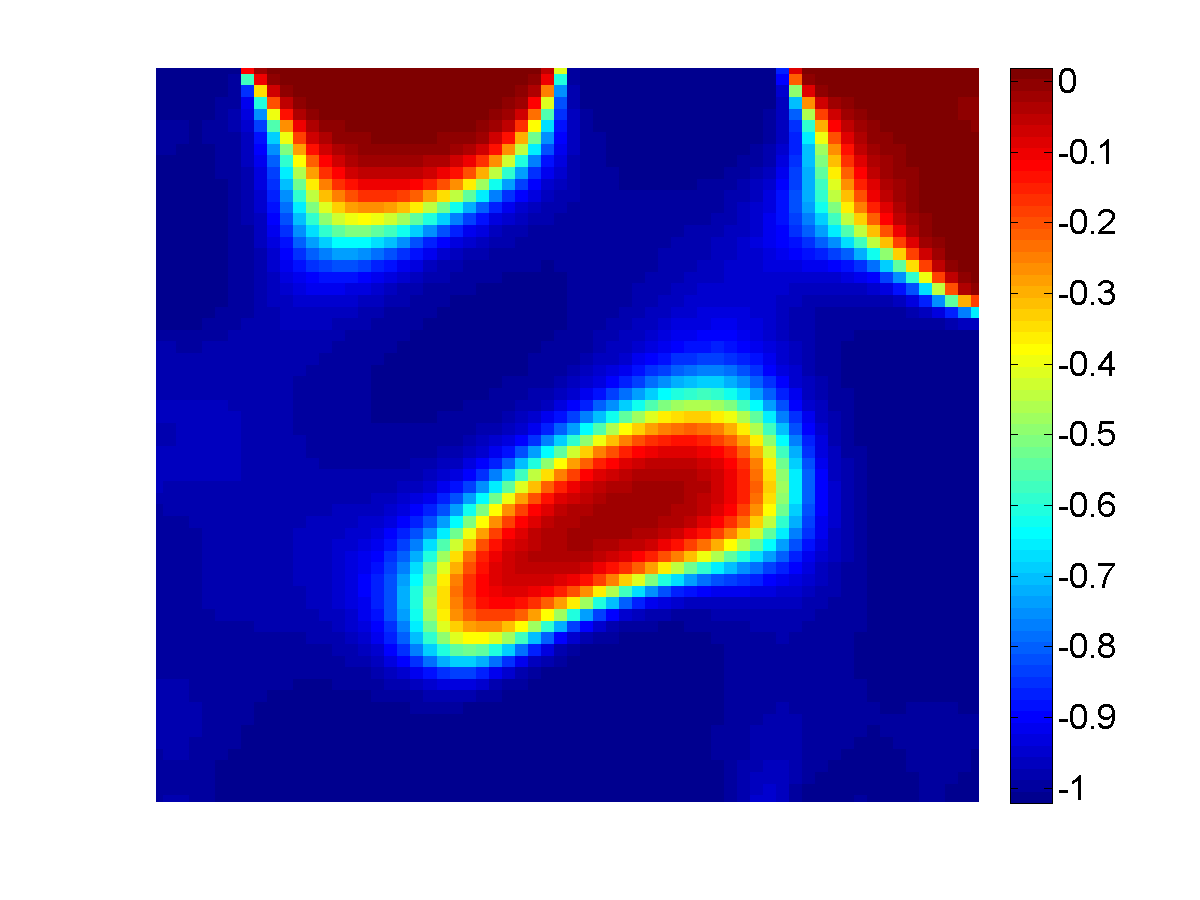}}
\subfigure[]{
\includegraphics[scale=0.15]{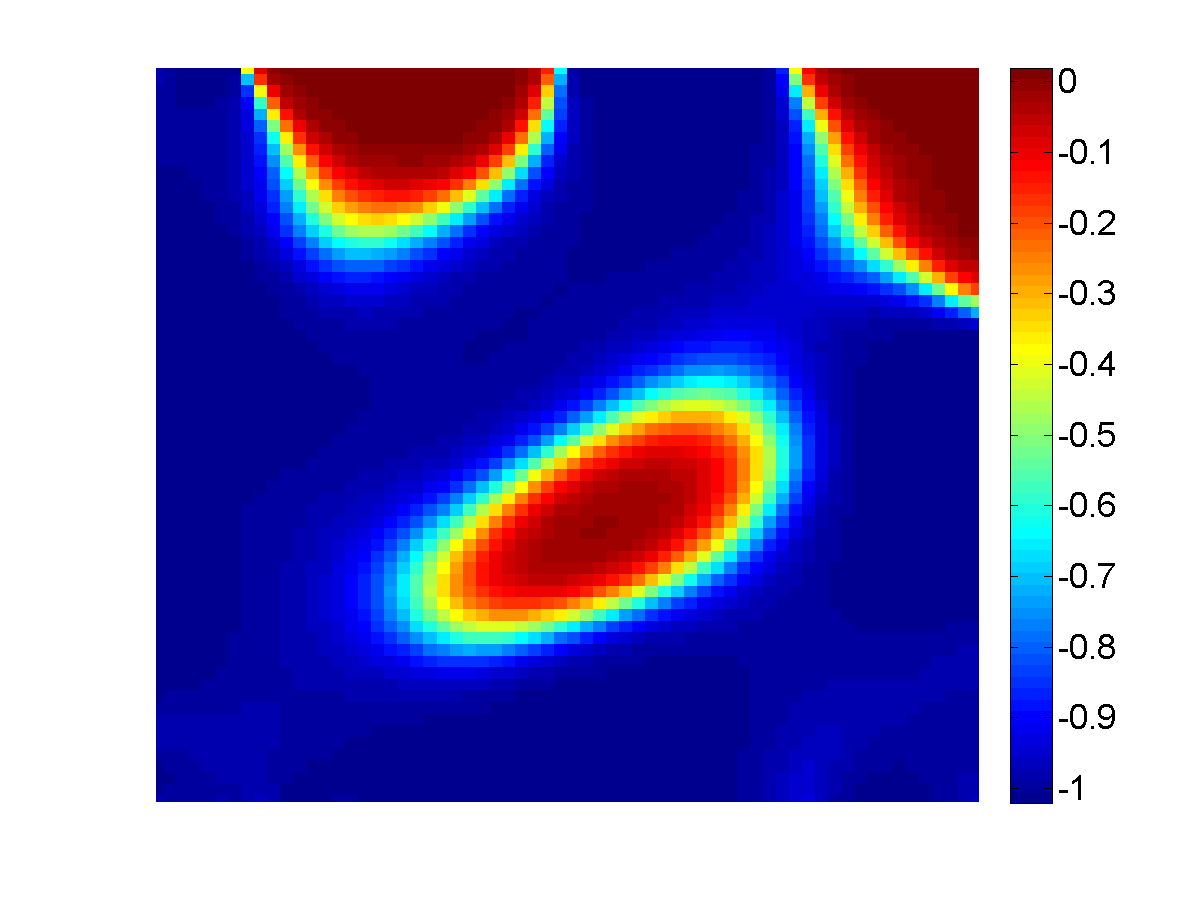}}
\subfigure[]{
\includegraphics[scale=0.15]{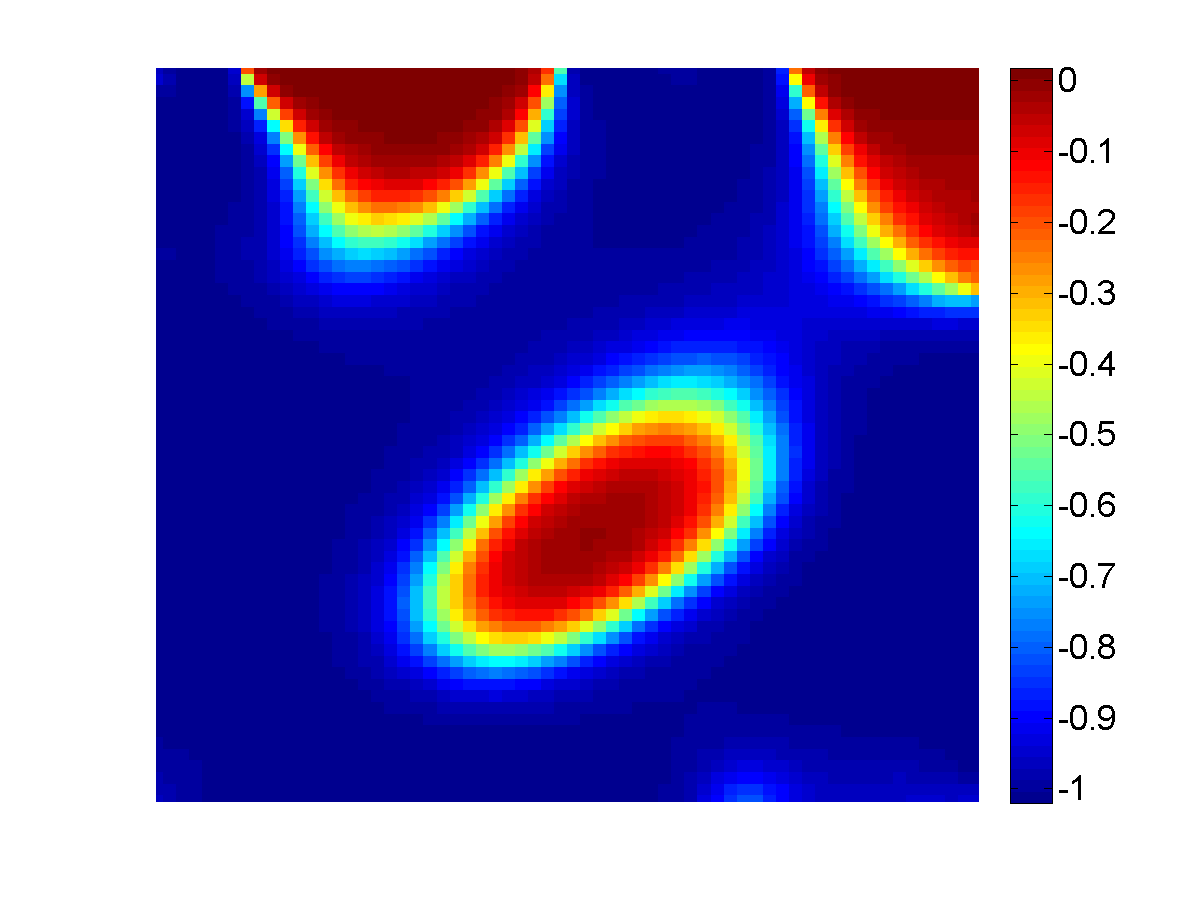}}
\caption{Example~\hyperref[numer_setup_a]{(E.1)}. Plots of log-conductivity of the recovered model using the 8 variants of Algorithm~\ref{alg1}, described in Section~\ref{sec:alg} and indicated here by (i)--(viii). The quality of reconstructions is generally comparable to that of plain vanilla with $s = 3,969$ and across variants.}
\label{fig:pde_example_recovery_2_level}
\end{figure}

For the calculations displayed here we have employed {\em dynamical regularization}~\cite{doas1,doas3}. 
In this method there is no explicit regularization term $R(\mm)$ in~\eqref{objective} 
and the regularization is done implicitly and iteratively.

The quality of reconstructions obtained by the various variants in Figure~\ref{fig:pde_example_recovery_2_level} 
is comparable to that of the ``vanilla'' with $s = 3,969$ in Figure~\ref{fig:pde_example_recovery_Vanilla_2_level}(b). 
In contrast, employing only $s = 49$ data sets
corresponding to similar experiments distributed over a coarser grid yields an inferior reconstruction
in Figure~\ref{fig:pde_example_recovery_Vanilla_2_level}(c). The cost of this latter run is $5,684$ PDE solves,
which is more expensive than our randomized algorithms for the much larger $s$.
Furthermore, comparing Figures~\ref{fig:pde_example_recovery_Vanilla_2_level}(b) and~\ref{fig:pde_example_recovery_2_level} to Figures~3 and~4 of~\cite{rodoas2}, which shows similar results for $s = 961$ data sets, we again see a relative improvement in reconstruction quality. All of this goes to show that
a large number of measurements $s$ can be crucial for better reconstructions.
Thus, it is not the case that one can dispense with a large portion of the measurements and 
still expect the same quality reconstructions. 
Hence, it is indeed useful to have algorithms such as Algorithm~\ref{alg1} that, 
while taking advantage of the entire available data, 
can efficiently carry out the computations and yet obtain credible reconstructions.

We have resisted the temptation to make comparisons between values of $\phi(\mm_{k+1})$ and $\hat \phi (\mm_{k+1})$
for various iterates. There are two major reasons for that. The first is that $\hat \phi$ values
in bounds such as~\eqref{cross_valid_hard},~\eqref{cross_valid_soft},~\eqref{uncert_check_hard},~\eqref{uncert_check_soft} and~\eqref{stop_crits} are different and are always compared
against tolerances in context
that are based on noise estimates. In addition, the sample sizes that we used for uncertainty check and stopping criteria, 
since they are given by Theorems~\ref{main_trace_theorem_lower} and~\ref{main_trace_theorem_upper}, 
already determine how far the estimated misfit is from the true misfit.
The second (and more important) reason  is that in such a highly diffusive forward problem as DC resistivity, misfit
values are typically far closer to one another than the resulting reconstructed models $\mm$ are.
A good misfit is merely a necessary condition, which can fall significantly short of being sufficient,
for a good reconstruction~\cite{haasol,rodoas2}. 

\subsubsection{Example~\hyperref[numer_setup_b]{(E.2)}}

Here we have imposed prior knowledge on the ``\textit{discontinuous}'' model in the form of 
total variation (TV) regularization~\cite{doasha,chta1,borsic2007total}. 
Specifically, $R(\mm)$ in~\eqref{objective} is the discretization of the TV functional 
$\int_{\Omega} | \grad m (\xx) |$. For each recovery, the regularization parameter, $\alpha$, 
has been chosen by trial and error within the range $[10^{-6},10^{-3}]$ to visually yield the best quality recovery. 

\begin{table}[!ht]
\begin{center}
\begin{tabular}{|ccccccccc|}
\hline 
Vanilla & (i) & (ii) & (iii) & (iv) & (v) & (vi) & (vii) & (viii)
\\ \hline 
476,280 & 5,631 & 5,057 & 5,011 & 3,990 & 6,364 & 4,618 & 4,344 & 4,195 \\ \hline
\end{tabular}
\end{center}
\caption{Example~\hyperref[numer_setup_b]{(E.2)}. Work in terms of number of PDE solves for all variants of Algorithm~\ref{alg1}, 
described in Section~\ref{sec:alg} and indicated here by (i)--(viii). 
The ``vanilla'' count is also given, as a reference. \label{table01_3level}}
\end{table}

Table~\ref{table01_3level} and Figures~\ref{fig:pde_example_recovery_Vanilla_3_level} 
and~\ref{fig:pde_example_recovery_3_level} tell a similar story as in Example~\hyperref[numer_setup_a]{(E.1)}. 
The quality of reconstructions with $s = 3,969$ by the various variants, displayed in 
Figure~\ref{fig:pde_example_recovery_3_level}, 
is comparable to that of the ``vanilla'' version  in Figure~\ref{fig:pde_example_recovery_Vanilla_3_level}(b), 
yet is obtained at only at a fraction (about 1\%) of the cost. 
The ``vanilla'' solution for $s=49$ displayed in Figure~\ref{fig:pde_example_recovery_Vanilla_3_level}(c), 
costs $5,978$ PDE solves, which again is 
a higher cost for an inferior reconstruction compared to our Algorithm~\ref{alg1}. 

\begin{figure}[htb]
\centering
\subfigure[]{\includegraphics[scale=0.15]{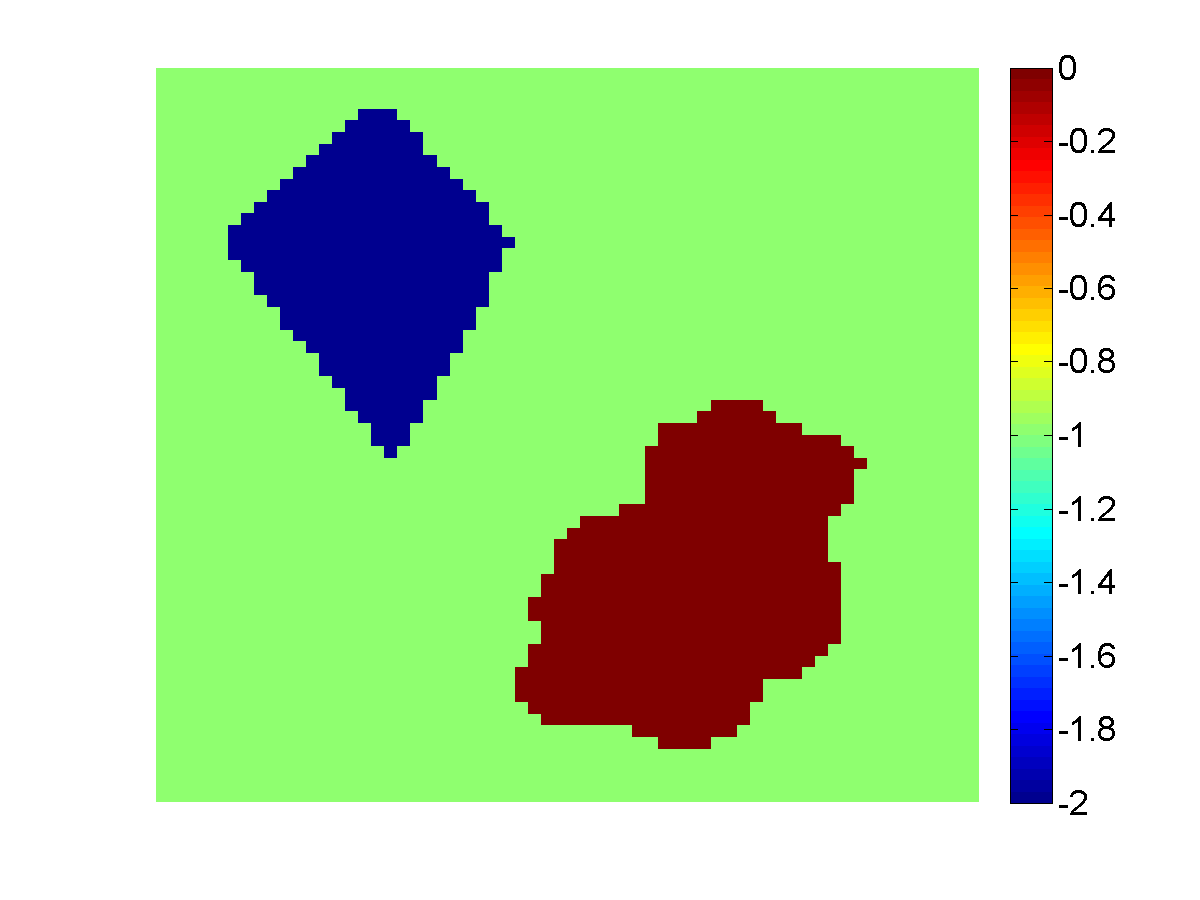}}
\subfigure[]{\includegraphics[scale=0.15]{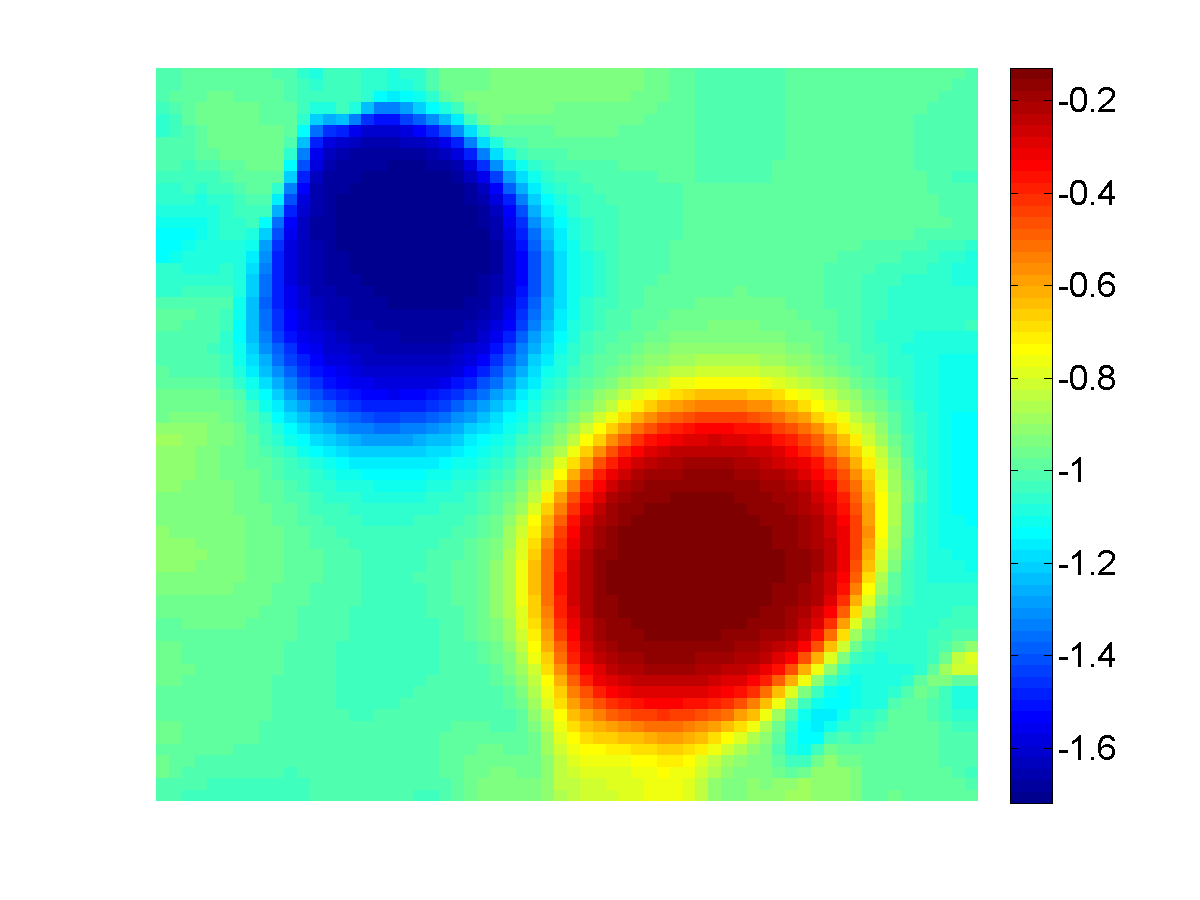}}
\subfigure[]{\includegraphics[scale=0.15]{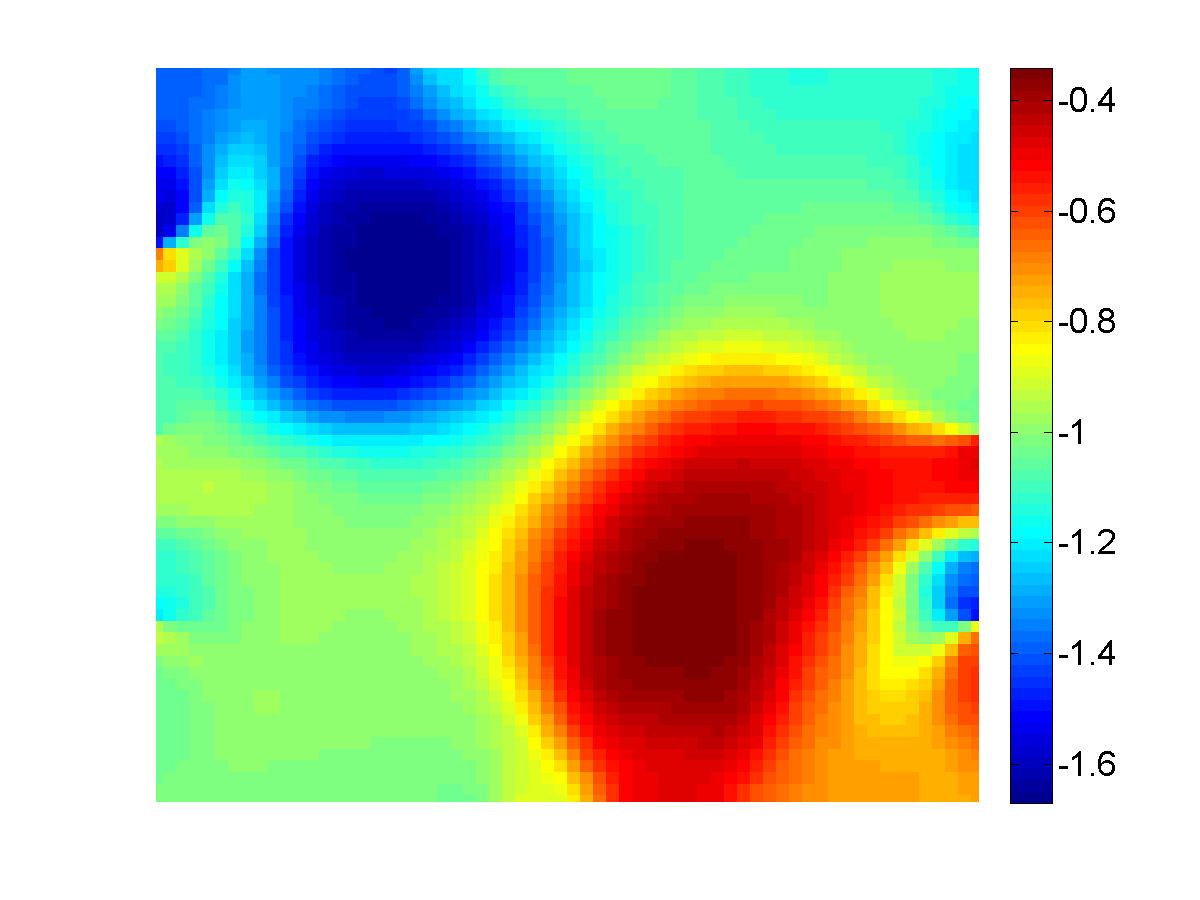}}
\caption{Example~\hyperref[numer_setup_b]{(E.2)}. Plots of log-conductivity: (a) True model; (b) Vanilla recovery with $s = 3,969$; (c) Vanilla recovery with $s=49$. 
The vanilla recovery using only $49$ measurement sets is clearly inferior, 
showing that a large number of measurement sets can be crucial for better reconstructions.}
\label{fig:pde_example_recovery_Vanilla_3_level}		
\renewcommand{\thesubfigure}{(\roman{subfigure})}
\subfigure[]{
\includegraphics[scale=0.15]{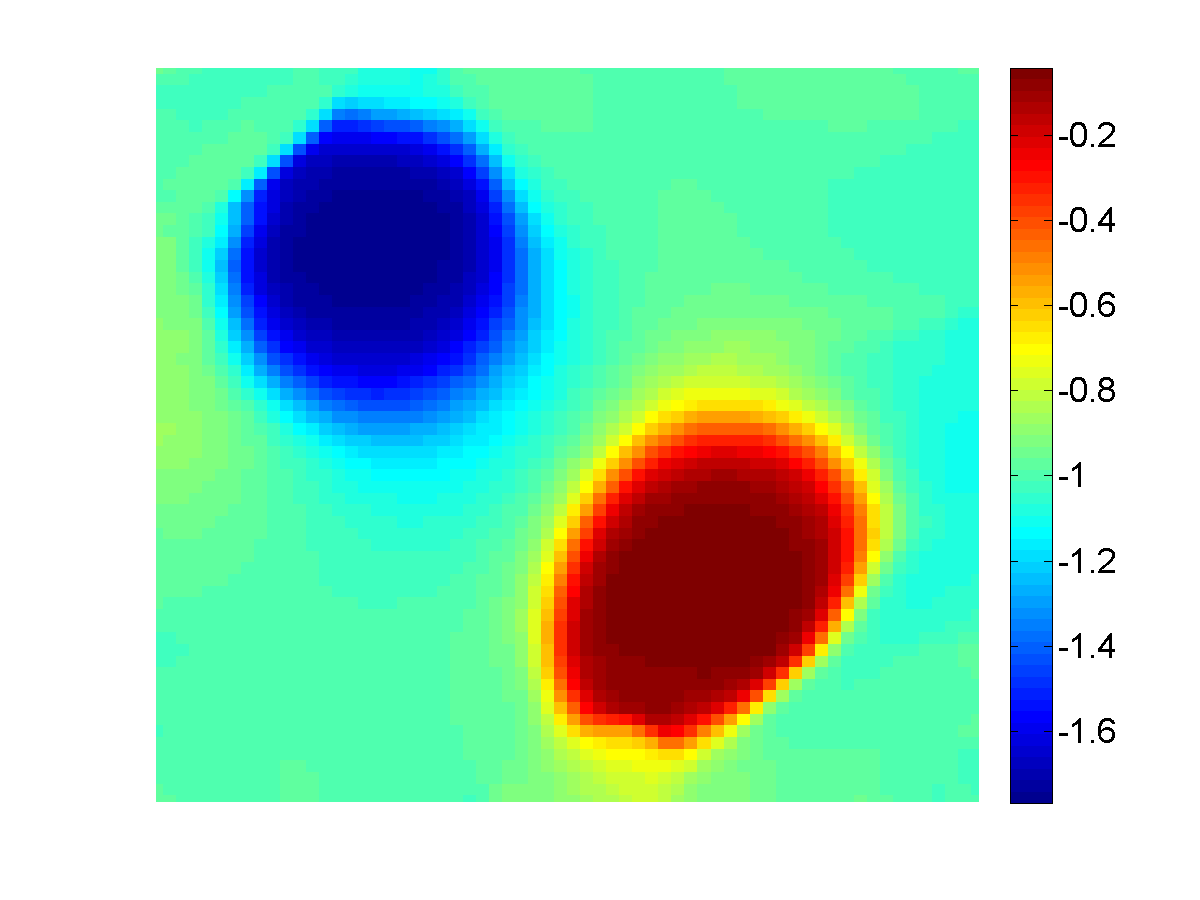}}
\subfigure[]{
\includegraphics[scale=0.15]{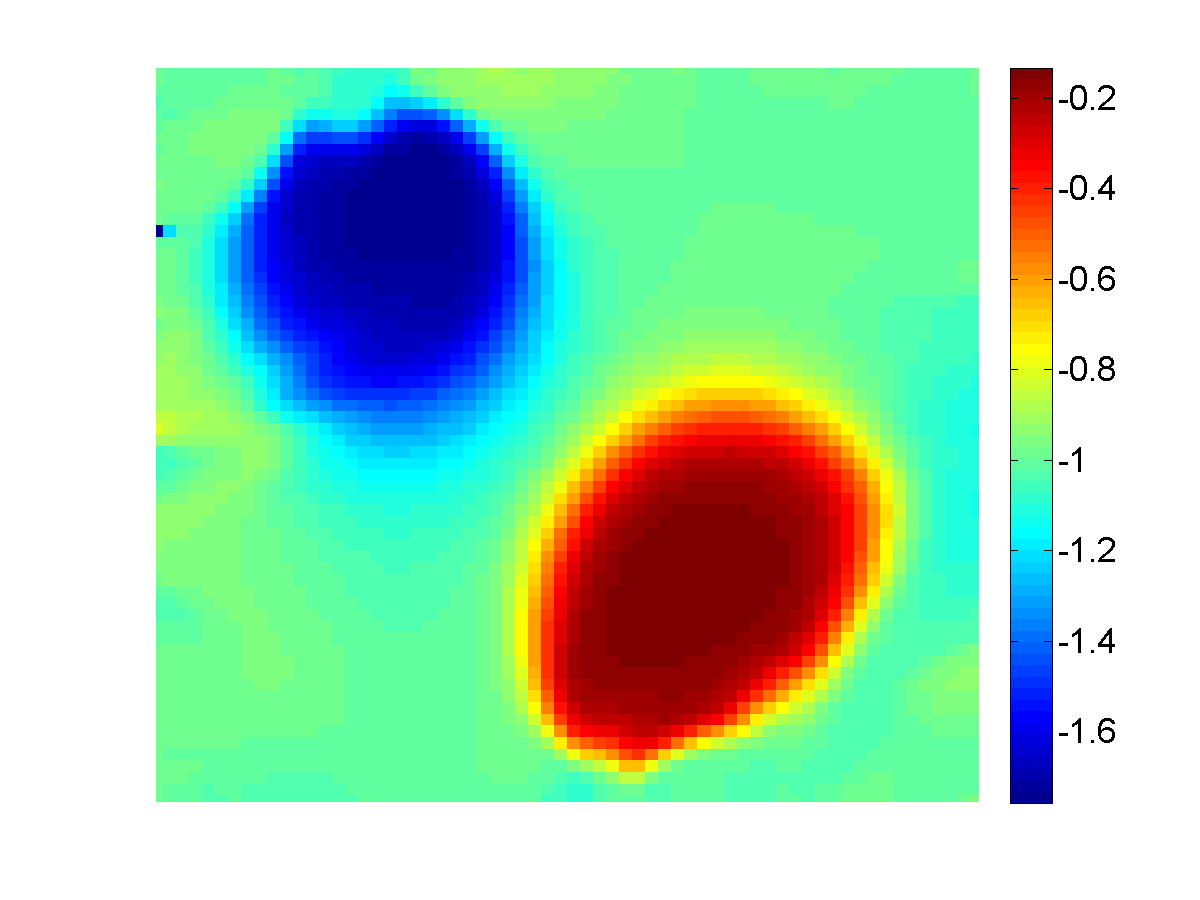}}
\subfigure[]{
\includegraphics[scale=0.15]{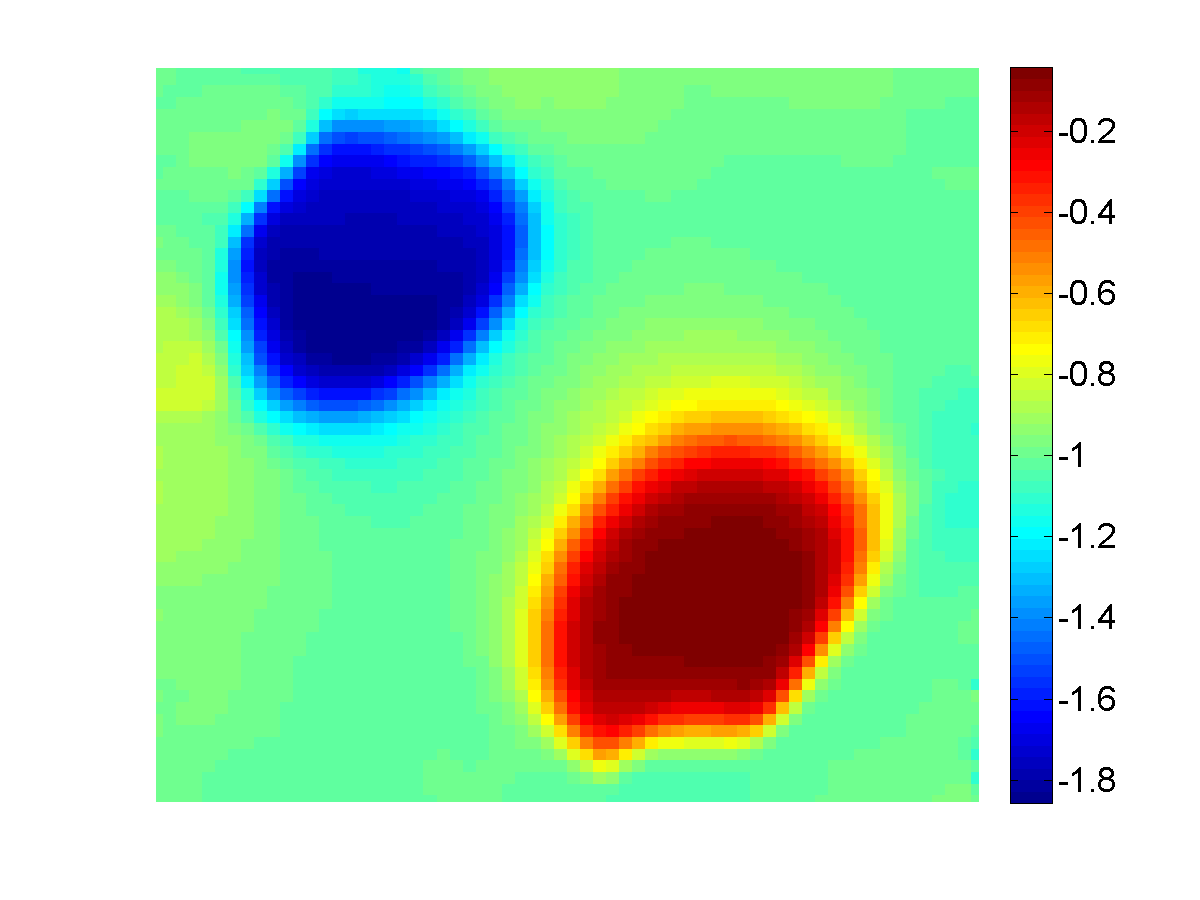}}
\subfigure[]{
\includegraphics[scale=0.15]{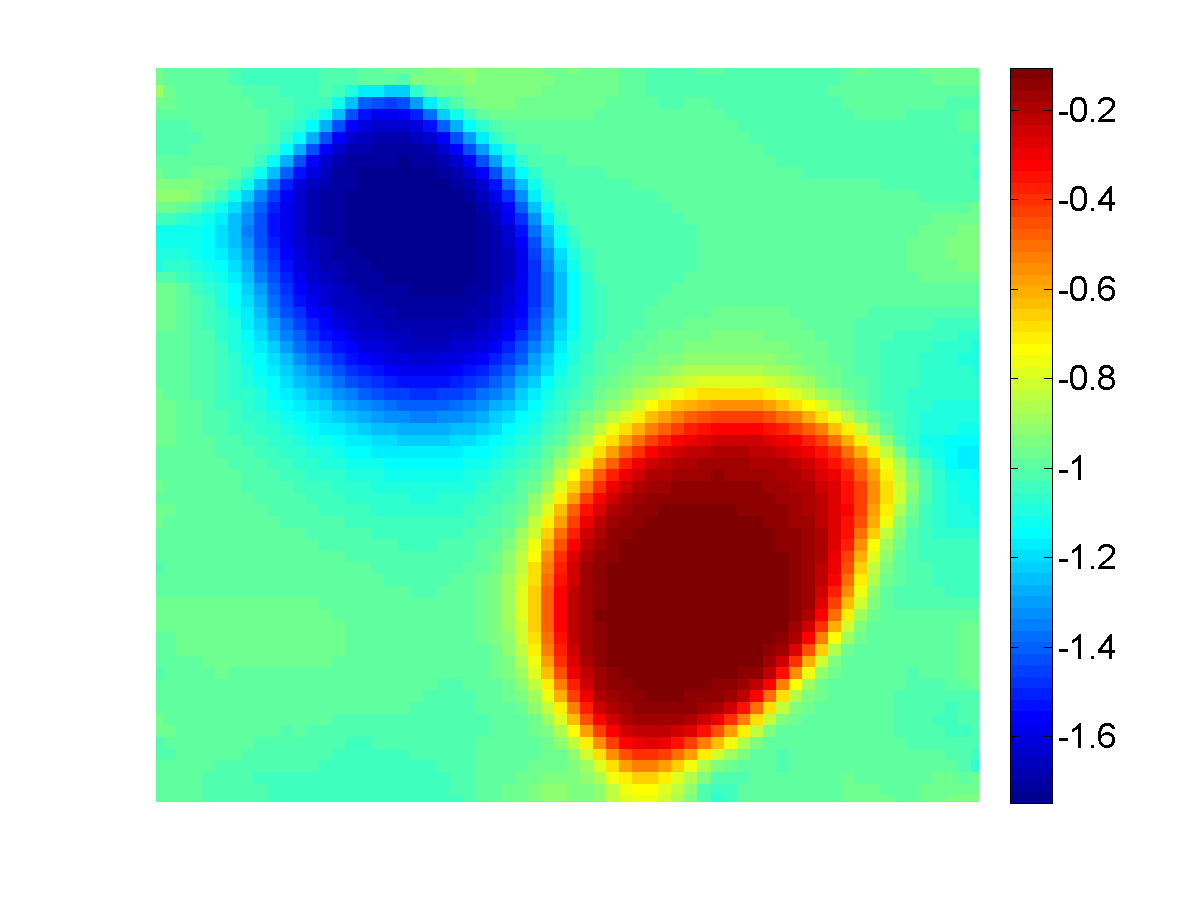}}
\subfigure[]{
\includegraphics[scale=0.15]{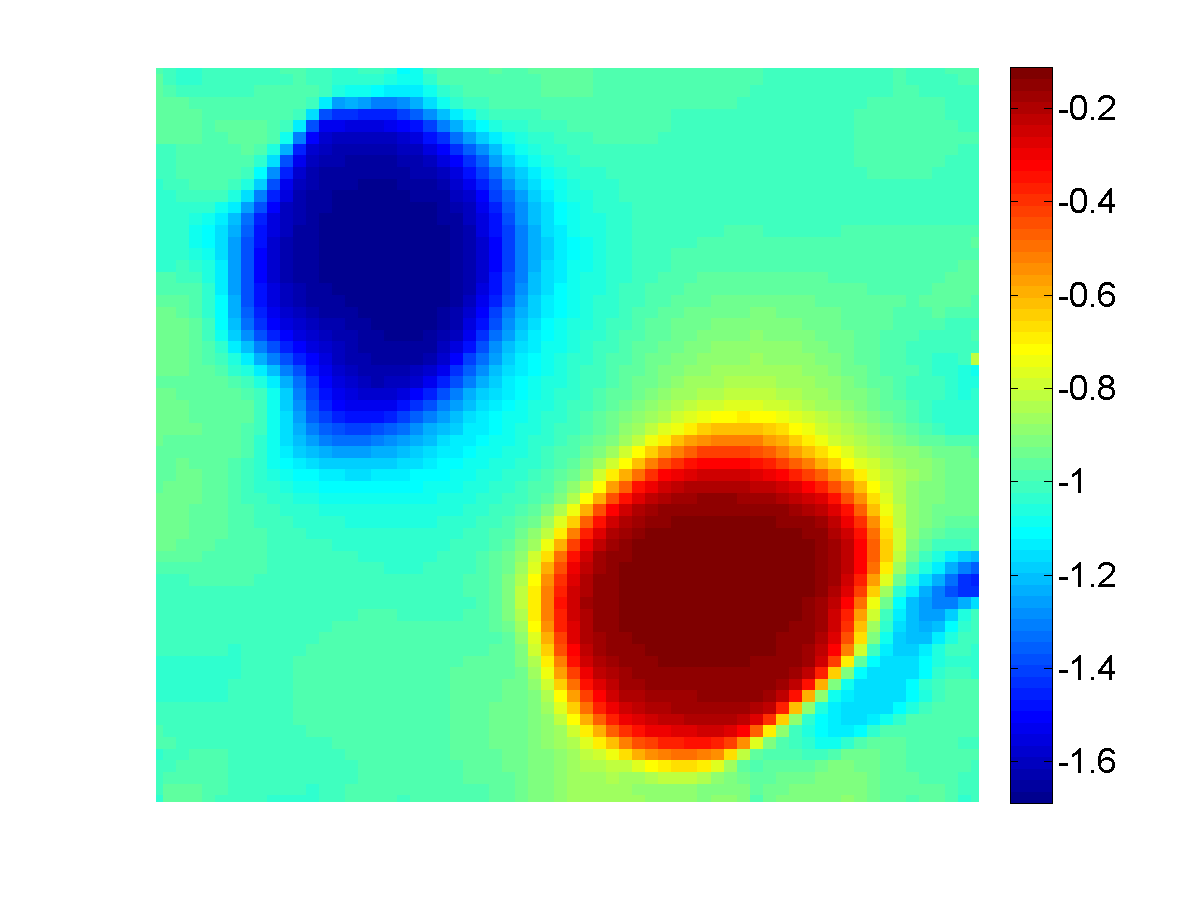}}
\subfigure[]{
\includegraphics[scale=0.15]{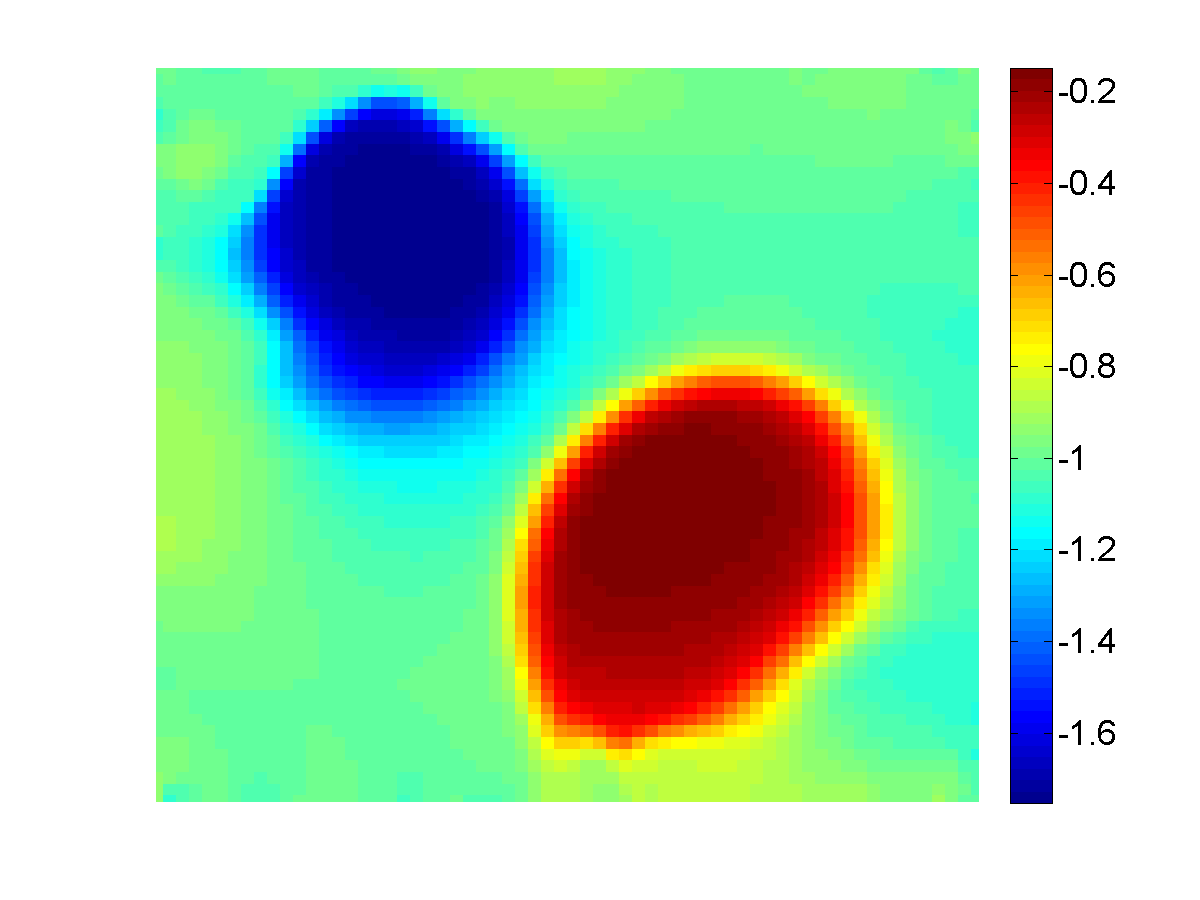}}
\subfigure[]{
\includegraphics[scale=0.15]{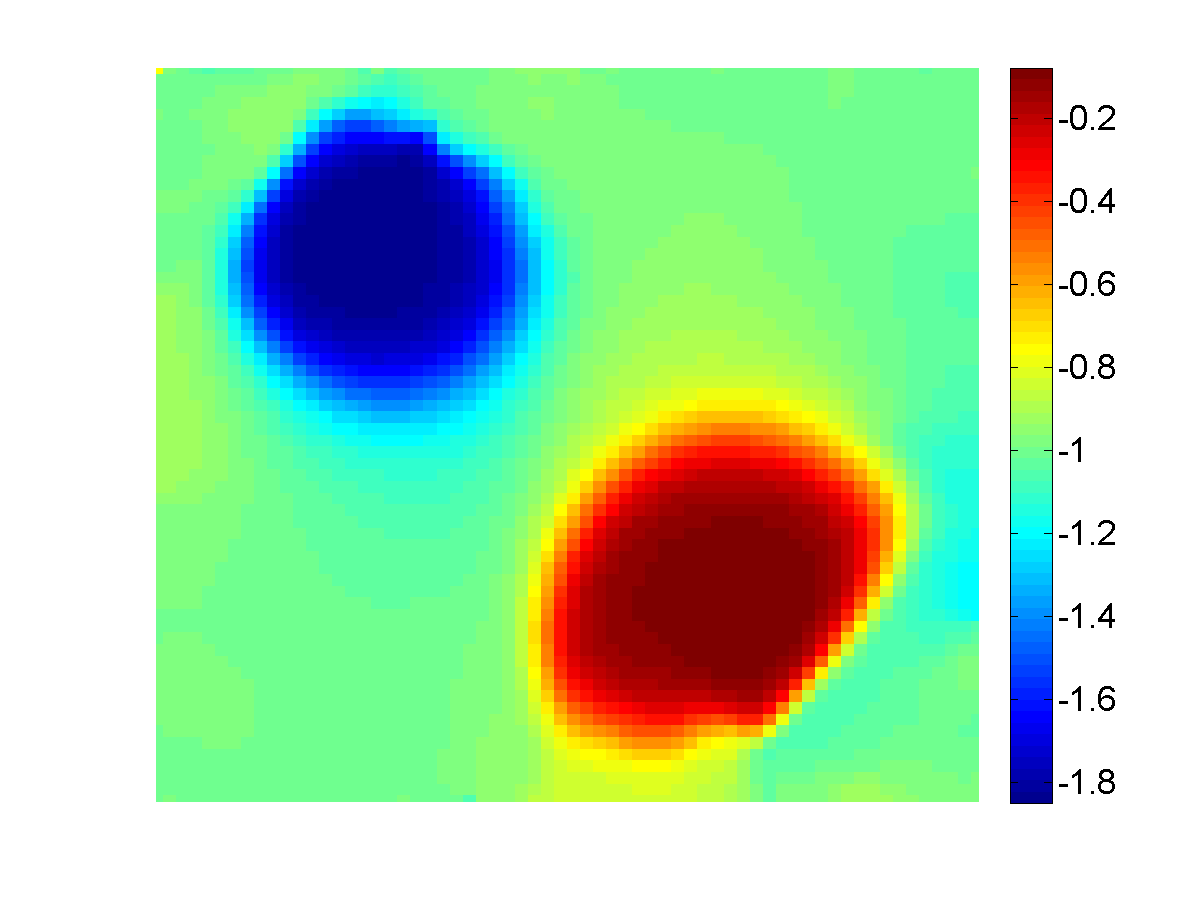}}
\subfigure[]{
\includegraphics[scale=0.15]{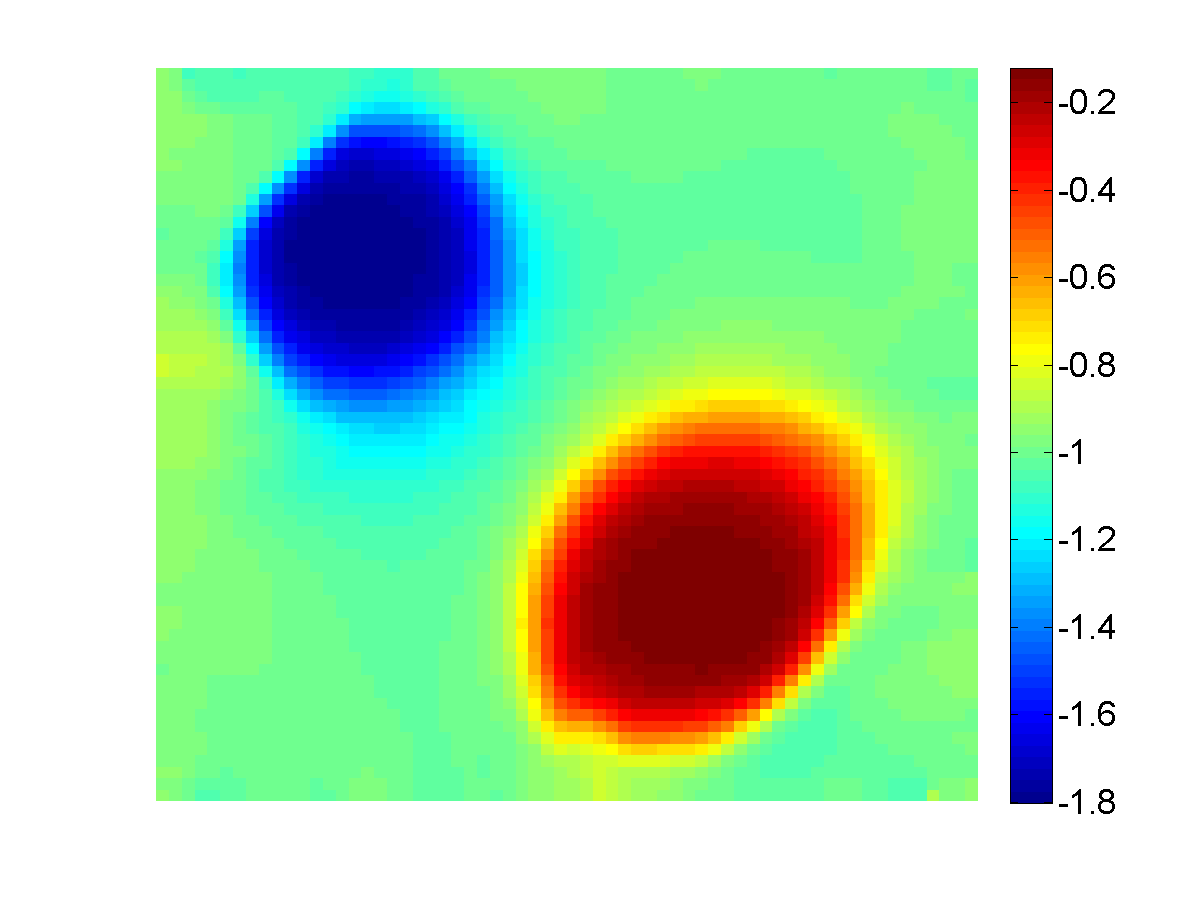}}
\caption{Example~\hyperref[numer_setup_b]{(E.2)}. Plots of log-conductivity of the recovered model using the 8 variants of Algorithm~\ref{alg1}, described in Section~\ref{sec:alg} and indicated here by (i)--(viii).
The quality of reconstructions is generally comparable to each other and that of plain vanilla with $s = 3,969$. }
\label{fig:pde_example_recovery_3_level}
\end{figure}

It is clear from Tables~\ref{table01_2level} and~\ref{table01_3level} that for most of these examples, 
variants (i)--(iv) which use the more aggressive cross validation~\eqref{cross_valid_hard} 
are at least as efficient as their respective counterparts, namely,
variants (v)--(viii) which use~\eqref{cross_valid_soft}. 
This suggests that, sometimes, a more aggressive sample size increase strategy may be a better option; 
see also the numerical examples in~\cite{rodoas1}. 
Notice that for all variants, the entire cost of the algorithm is comparable to one single evaluation 
of the misfit function $\phi(\mm)$ using the full data set! 


\begin{figure}[htb]
\centering
\mbox{
\includegraphics[scale=0.55]{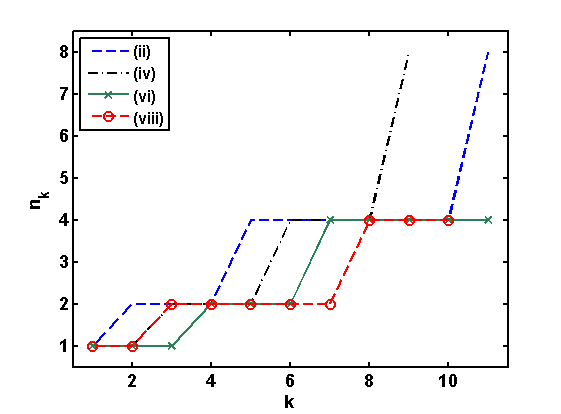}
}
\caption{Example~\hyperref[numer_setup_b]{(E.2)}. Growth of the fitting sample size, $n_{k}$, as a function of the iteration $k$,
upon using cross validation strategies~\eqref{cross_valid_hard} and~\eqref{cross_valid_soft}. The graph shows the fitting sample size growth for variants (ii) and (vi) of Algorithm~\ref{alg1}, as well as their counterparts, namely, variants (vi) and (viii). Observe that for variants (ii) and (iv) where~\eqref{cross_valid_hard} is used, 
the fitting sample size grows at a more aggressive rate than for variants (vi) and (viii) where~\eqref{cross_valid_soft} is used.}
\label{fig:pde_example_sample_size}
\end{figure}

\section{Conclusions}
\label{sec:conc}
In the present article we have proved tight necessary and sufficient conditions 
for the sample size, $n$, required to reach, with a probability of at least $1 - \delta$, 
(one-sided) approximations for $tr(A)$ to within a relative tolerance $\veps$. 
All of the sufficient conditions are computable in practice and do not assume any a priori knowledge about the matrix. 
If the rank of the matrix is known then the necessary bounds can also be computed in practice. 

Subsequently, using these conditions, we have presented eight variants of a general purpose algorithm for solving 
an important class of large scale non-linear least squares problems. 
These algorithms can be viewed as an extended version of those in~\cite{rodoas1, rodoas2}, 
where the uncertainty in most of the stochastic steps is quantified. 
Such uncertainty quantification allows one to have better control over the behavior of the algorithm 
and have more confidence in the recovered solution. 
The resulting algorithm is presented in Section~\ref{sec:alg}. 

Furthermore, we have demonstrated the performance of our algorithm using an important class of
problems which arise often in practice, namely, PDE inverse problems with many measurements. 
By examining our algorithm in the context of the DC resistivity problem 
as an instance of such class of problems, 
we have shown that Algorithm~\ref{alg1} can recover solutions with remarkable efficiency. 
This efficiency is comparable to similar heuristic algorithms proposed in~\cite{rodoas1, rodoas2}. 
The added advantage here is that with the uncertainty being quantified, 
the user can have more confidence in the approximate solution obtained by our algorithm. 

Tables~\ref{table01_2level} and~\ref{table01_3level} show the amount of work (in PDE solves) of the 8 variants of our algorithm.
Compared to a similar algorithm which uses the entire data set, an efficiency improvement by two orders of magnitude
is observed.
For most of the examples considered, the same tables also show that the more aggressive cross validation 
strategy~\eqref{cross_valid_hard} is, at least, as efficient as the more relaxed strategy~\eqref{cross_valid_soft}. 
A thorough comparison of the behavior of these cross validation strategies (and all of the variants, in general) 
on different examples and model problems is left for future work.


 

\noindent{\bf Acknowledgment}
 We thank our anonymous referees for several valuable comments which have helped to improve the text. The first author thanks Prof.\ Yaming Yu for referring him to~\cite{szba}, which resulted in the collaboration among the authors of the present paper.

\appendix

\section{Numerical experiments setup}
\label{numerical_setup}

The experimental setting we use in Section~\ref{sec:pde_manys} is as follows: 
for each experiment $i$, there is a positive unit point source at $\xx^{i}_1$ and a negative sink 
at $\xx^{i}_{2}$, where $\xx^{i}_{1}$ and $\xx^{i}_{2}$ denote two locations on the 
boundary $\partial \Omega$. 
Hence in~\eqref{2.1a} we must consider sources of the form 
$q_i(\xx) = \delta(\xx-\xx_{1}^i) - \delta(\xx-\xx_{2}^i)$, 
i.e., a difference of two $\delta$-functions, and $\qq_{i}$ is the discretization of $q_{i}$ over the grid.

For our experiments, when we place a source on the left boundary, 
we place the corresponding sink on the right boundary in every possible combination. 
Hence, having $p$ locations on the left boundary for the source would result in $s = p^2$ experiments.
The receivers are located at the top and bottom boundaries. No source or receiver is placed at the corners.

We then generate data $\dd_i$ by using a chosen true model (or ground truth)
and a source-receiver configuration as described above. 
This is followed by peppering these values with $2\%$ additive Gaussian noise to create the data $\dd_i$
used in our experiments.
Specifically, for an additive noise of $2\%$, denoting the ``clean data'' $l \times s$ matrix by $D^*$,
we reshape this matrix into a vector $\dd^*$ of length $sl$, calculate the standard deviation 
${\tt \sigma} = .02\| \dd^* \|/\sqrt{sl}$, and define $D = D^* + {\tt \sigma * randn(l,s)}$
using {\sc Matlab}'s random generator function {\tt randn}. Following the celebrated Morozov {\em discrepancy principle}~\cite{vogelbook,ehn1,morozov,kaso}, the stopping tolerance is set to be $\rho = \tau \sigma^2 sl$. As in~\cite{rodoas1}, we choose $\tau = 1.2$.

For all numerical experiments, in order to avoid committing ``inverse crime'', the ``true field'' is calculated on a grid that is twice as fine as the 
one used to reconstruct the model. For the 2D examples, the reconstruction is done on a uniform grid of size $64^2$
with $s = 3,969$ experiments in the setup described above.

As for an iterative method to decrease the value of the objective function, we employ variants of stabilized GN; see~\cite[Appendix A]{rodoas1} for more details. At each iteration of such method, an update direction needs to be calculated. Usually another iterative scheme is used to calculate the update. We employ preconditioned conjugate gradient (PCG) as our inner iterative solver. The PCG iteration limit is set to $20$, and the PCG tolerance was chosen to be $10^{-3}$. We again refer to~\cite[Appendix A]{rodoas1} for more details. The initial guess for GN iterations is $\mm_{0} = {\bf 0}$.  

For the transfer function $\psi$ described in Section~\ref{sec:pde_manys}, 
we use the formulation~\cite[Eqn.\ (6.3)]{rodoas1} with 
$\mu_{\max} = 1.2 \max \mu(\xx)$, and $\mu_{\min} = .83 \min \mu(\xx)$.

\section{Extremal probabilities of linear combinations of gamma random variables}
\label{sec:appendix}

In this appendix we prove Theorems~\ref{monotonicity_gamma_theorem} and~\ref{extremal_prob_thm}.
Such results were obtained in~\cite{szba} for the special case where the $X_{i}$'s are chi-squared r.v's 
of degree 1 (corresponding to $\alpha = \beta = 1/2$). 
Here we extend those results to 
arbitrary gamma random variables, including
chi-squared of arbitrary degree, exponential, Erlang, etc.

In what follows,
for a gamma r.v $X \sim Gamma(\alpha,\beta)$, we use the notation $f_{X}$ for its probability density function (PDF) and
$F_{X}$ for its cumulative distribution function (CDF).
The objective in the proof of Theorem~\ref{extremal_prob_thm} is to find the extrema (with respect to $\bm{\lambda} \in \Theta$) of the CDF of r.v $\sum_{i=1}^{n} \lambda_{i} X_{i}$. 
This is mainly achieved by perturbation arguments, 
employing a key identity which is derived using Laplace transforms. 
Using our perturbation arguments with this identity and employing Lemma~\ref{lemma_1*}, 
we obtain that at any extremum, we must have either 
$\lambda_{1}, \lambda_{2} > 0$ and $\lambda_{3} = \cdots = \lambda_{n} = 0$ or for some 
$i \leq n$ we must get $\lambda_{1} = \cdots = \lambda_{i} > 0$ and $\lambda_{i+1} = \cdots = \lambda_{n} = 0.$ (Note that this latter case covers the ``corners'' as well.). 
In the former case, Lemma~\ref{lemma_2*} is used to distinguish between the minima and maxima for different values of $x$. 
These results along with Theorem~\ref{monotonicity_gamma_theorem} are then used to prove Theorem~\ref{extremal_prob_thm}. 

Three lemmas are used in the proofs of our two theorems.
	Lemma~\ref{lemma_A*} describes some properties of the 
	PDF of non-negative linear combinations of arbitrary gamma r.v's, such as analyticity and vanishing derivatives at zero. 
	Lemma~\ref{lemma_1*} describes the monotonicity property of the mode of the PDF of non-negative linear combinations 
	of a \textit{particular} set of gamma r.v's, which is useful for the proof of Theorem~\ref{extremal_prob_thm}.
	Lemma~\ref{lemma_2*} gives some properties regarding the mode of the PDF of convex combinations of two \textit{particular} gamma r.v's, which is used in proving Theorem~\ref{monotonicity_gamma_theorem} and Theorem~\ref{extremal_prob_thm}.


\subsection{Lemmas}
\label{additional_Lemmas}
We next state and prove the lemmas summarized above.

\begin{lemma}[{Generalization of~\cite[Lemma A]{szba}}]
Let $X_{i} \sim	Gamma(\alpha_{i},\beta_{i}),\;  i=1,2,\ldots,n,$ be independent r.v's, where $\alpha_{i}, \beta_{i} > 0 \; \forall i$. Define $Y_{n} \defeq \sum_{i=1}^{n} \lambda_{i} X_{i}$ for $\lambda_{i} > 0$, $\forall i$ and $\rho_{j} \defeq \sum_{i=1}^{j} \alpha_{i}$.
Then for the PDF of $Y_{n}$, $f_{Y_{n}}$, we have
\begin{enumerate}[(i)]
	\item $f_{Y_{n}} > 0$, $\forall x > 0$,
	\item $f_{Y_{n}}$ is analytic on $\mathbb{R}^{+} = \{x | x > 0\}$,
	\item $f_{Y_{n}}^{(k)}(0) = 0$, if $0 \leq k < \rho_{n} - 1$, where $f_{Y_{n}}^{(k)}$ denotes the $k^{th}$ derivative of $f_{Y_{n}}$.
\end{enumerate} 
\label{lemma_A*}
\end{lemma}
\begin{proof}
The proof is done by induction on $n$. For $n=2$ we have
\begin{eqnarray*}
f_{Y_{2}}(x) &=& \int_{0}^{\infty} f_{\lambda_{1} X_{1}}(y) f_{\lambda_{2} X_{2}}(x-y) dy \\
&=& \frac{(\beta_{1}/\lambda_{1})^{\alpha_{1}} (\beta_{2}/\lambda_{2})^{\alpha_{2}} }{\Gamma(\alpha_{1})\Gamma(\alpha_{2})} \int_{0}^{x} y^{\alpha_{1}-1} (x-y)^{\alpha_{2}-1} e^{- \frac{\beta_{1}y}{\lambda_{1}} - \frac{\beta_{2} (x-y) }{\lambda_{2}}} dy.
\end{eqnarray*}
Now the change of variable $y \rightarrow x \cos^{2}\theta_{1}$ would yield
\begin{equation*}
f_{Y_{2}}(x) = 2 \frac{(\beta_{1}/\lambda_{1})^{\alpha_{1}} (\beta_{2}/\lambda_{2})^{\alpha_{2}} }{\Gamma(\alpha_{1})\Gamma(\alpha_{2})} x^{(\alpha_{1}+\alpha_{2}-1)} \int_{0}^{\frac{\pi}{2}} (\cos\theta_{1})^{2\alpha_{1}-1} (\sin\theta_{1})^{2\alpha_{2}-1}e^{- x (\frac{\beta_{1} \cos^{2}\theta_{1}}{\lambda_{1}} + \frac{\beta_{2} \sin^{2}\theta_{1} }{\lambda_{2}})} d\theta_{1}.
\end{equation*}
By induction on $n$, one can show that for arbitrary $n \geq 2$
\begin{subequations}
\begin{equation}
f_{Y_{n}}(x) = 2^{n-1} \left(\prod_{i=1}^{n} \frac{(\beta_{i}/\lambda_{i})^{\alpha_{i}}}{\Gamma(\alpha_{i})}\right) x^{\rho_{n} - 1} \int_{D^{n-1}}  P_{n}( \Theta_{n-1}) Q_{n}(\Theta_{n-1}) e^{-x R_{n}(\Theta_{n-1})} \dd \Theta_{n-1},
\label{induct_pdf}
\end{equation}
where 
\begin{equation}
P_{n}(\Theta_{n-1}) = \prod_{j=1}^{n-1} (\cos\theta_{j})^{2\rho_{j}-1}, \quad Q_{n}(\Theta_{n-1}) = \prod_{j=1}^{n-1} (\sin\theta_{j})^{2\alpha_{j+1}-1},
\label{P_and_Q}
\end{equation}
the function $R_{n}(\Theta_{n-1})$ satisfies the following recurrence relation
\begin{eqnarray}
R_{n}(\Theta_{n-1}) &=& \cos^{2}\theta_{n-1} R_{n-1}(\Theta_{n-2}) + \beta_{n}\lambda_{n}^{-1} \sin^{2}\theta_{n-1}, \quad \forall n \geq 2 \\
R_{1}(\Theta_{0}) &=& \beta_{1}/\lambda_{1},
\label{recur_R}
\end{eqnarray}
and $\dd\Theta_{n-1}$ denotes the $n-1$ dimensional Lebesgue measure with the domain of integration
\begin{equation}
D^{n-1} = (0,\pi/2) \times (0,\pi/2) \times \ldots \times (0,\pi/2) = (0,\pi/2)^{n-1} \subset \mathbb{R}^{n-1}.
\label{domain}
\end{equation}
\label{main_eqn_lemma_A*}
\end{subequations}
Now the claims in Lemma~\ref{lemma_A*} follow from~\eqref{main_eqn_lemma_A*}.
$\blacksquare$
\end{proof}

\input exline
\begin{lemma}[{Generalization of~\cite[Lemma 1]{szba}}]
Let $X_{i} \sim	Gamma(\alpha_{i},\alpha), \; i=1,2,\ldots,n,$ be independent r.v's, where $\alpha_{i} > 0 \; \forall i$ and $\alpha > 0$. Also let $\psi \sim Gamma(1,\alpha)$ be another r.v independent of all $X_{i}$'s. If $\sum_{i=1}^{n} \alpha_{i} > 1$, then the mode, $\bar{x}(\lambda)$, of the r.v 
$W({\lambda}) = Y + \lambda \psi$
is strictly increasing in $\lambda>0$, where $Y = \sum_{i=1}^{n} \lambda_{i} X_{i}$ with $\lambda_{i} > 0$, $\forall i$.
\label{lemma_1*}
\end{lemma}
\begin{proof}
The proof is almost identical to that of Lemma 1 in~\cite{szba}; hence, we omit the details.
$\blacksquare$
\end{proof}

\input exline
\begin{lemma}[{Generalization of~\cite[Lemma 2]{szba}}]
For some $ \alpha_{2} \geq \alpha_{1} > 0$, let $\xi_{1} \sim Gamma(1+\alpha_{1},\alpha_{1})$ and $\xi_{2} \sim Gamma(1+\alpha_{2}, \alpha_{2})$ be independent gamma r.v's. Also let $\bar{x} = \bar{x}(\lambda)$ denote the mode of the r.v $\xi(\lambda) = \lambda \xi_{1}  + (1- \lambda) \xi_{2}$ for $0 \le \lambda \le 1$. Then
\begin{enumerate}[(i)]
	\item for a given $\lambda$, $\bar{x}(\lambda)$ is unique, \label{lemma_2*_01}
	\item $1 \le  \bar{x}(\lambda)   \le  \frac {2  \sqrt{\alpha_{1}  \alpha_{2}} + 1}{2 \sqrt{\alpha_{1}  \alpha_{2}} }, \quad \forall 0 \le \lambda \le 1$, with $\bar{x}(0) = \bar{x}(1) = 1$ and, in case of $\alpha_{i} = \alpha_{j} = \alpha$, $\bar{x}(\hf) = \frac {2 \alpha + 1}{2 \alpha }$, otherwise the inequalities are strict, and \label{lemma_2*_02}
	\item there is a $\lambda^{*} \in \big(\frac{\sqrt{\alpha_{1}}}{\sqrt{\alpha_{2}}+\sqrt{\alpha_{1}}},1\big)$ such that the mode $\bar{x}(\lambda)$ is a strictly increasing function of $\lambda$  on $(0, \lambda^{*})$ and it is a strictly decreasing function on $(\lambda^{*} , 1)$ \label{lemma_2*_03} and, for $\alpha_{1}=\alpha_{2}$, we have $\lambda^{*} = \frac{1}{2}$.
\end{enumerate}
\label{lemma_2*}
\end{lemma}
\begin{proof}
Uniqueness claim~\eqref{lemma_2*_01} has already been proven in~\cite[Theorem 4]{szba}. 
We prove~\eqref{lemma_2*_03} since~\eqref{lemma_2*_02} is implied from within the proof. For $0 < \lambda < 1$, the PDF of $\xi(\lambda)$ can be written as
\begin{equation*}
f_{\xi(\lambda)}(x) = \int_{0}^{x} f_{\lambda \xi_{1}}(y) f_{(1-\lambda) \xi_{2}}(x-y) dy.
\end{equation*}
Since $f_{\lambda \xi_{1}}(0) = f_{(1-\lambda) \xi_{2}}(0) = 0$ we have
\begin{eqnarray*}
\frac{\partial}{\partial x} f_{\xi(\lambda)}(x) &=& \int_{0}^{x} f_{\lambda \xi_{1}}(y) \frac{\partial}{\partial x} f_{(1-\lambda) \xi_{2}}(x-y) dy \\
&=& -\int_{0}^{x} f_{\lambda \xi_{1}}(y) \frac{\partial}{\partial y} f_{(1-\lambda) \xi_{2}}(x-y) dy\\
&=& \int_{0}^{x} \frac{\partial}{\partial y} \left(f_{\lambda \xi_{1}}(y)\right)  f_{(1-\lambda) \xi_{2}}(x-y) dy
\end{eqnarray*}
where for the second equality we use the fact that $\frac{\partial}{\partial x} f(x-y) = -\frac{\partial}{\partial y} f(x-y)$, and for the third equality we used integration by parts. Let $\alpha = \alpha_{1}$ and $\alpha_{2} = c \alpha$ for some $c \geq 1$. So now we have
\begin{eqnarray*}
&& \frac{\partial}{\partial x} f_{\xi(\lambda)}(x) = \frac{(\frac{\alpha}{\lambda})^{1+\alpha} (\frac{c\alpha}{1-\lambda})^{1+\alpha c} }{\Gamma(1+\alpha)\Gamma(1+\alpha c)} \int_{0}^{x} \frac{\partial \left(y^{\alpha} e^{- \frac{\alpha y}{\lambda} }\right) }{\partial y} (x-y)^{\alpha c} e^{- \frac{c \alpha (x-y) }{1-\lambda}} dy\\
&=& \frac{\alpha^{2+\alpha} (c\alpha)^{1+c\alpha}}{\Gamma(1+\alpha) \Gamma(1+c \alpha)} \lambda^{-2-\alpha} \left(1-\lambda\right)^{-1-\alpha c} e^{-\frac{c \alpha x }{(1-\lambda)}} \int_{0}^{x} (\lambda - y) y^{\alpha-1} (x-y)^{\alpha c} e^{-\alpha y \left(\frac{1}{\lambda} - \frac{c}{1-\lambda}\right)} dy \\
&=& C(x,\lambda) A(x,\lambda),
\end{eqnarray*}
where 
\begin{eqnarray*}
C(x,\lambda) &\defeq& \frac{\alpha^{2+\alpha} (c\alpha)^{1+c\alpha}}{\Gamma(1+\alpha) \Gamma(1+c \alpha)} \lambda^{-2-\alpha} \left(1-\lambda\right)^{-1-\alpha c} e^{-\frac{c \alpha x }{(1-\lambda)}}, \\
A(x,\lambda) &\defeq& \int_{0}^{x} \left(\lambda - y\right) y^{\alpha-1} \left(x-y\right)^{\alpha c} e^{-\phi(\lambda) y} dy, \quad \\
\phi(\lambda) &\defeq& \alpha \left(\frac{1}{\lambda} - \frac{c}{1-\lambda}\right).
\end{eqnarray*}
Now if $\bar{x}$ is the mode of $\xi(\lambda)$, then we have
\begin{equation*}
\frac{\partial}{\partial x} f_{\xi(\lambda)}(\bar{x}) = C(\bar{x},\lambda) A(\bar{x},\lambda) = 0,
\end{equation*}
which implies that $A(\bar{x},\lambda) = 0$ since $C(\bar{x},\lambda) > 0$. Let us define the linear functional $L:\mathcal{G} \rightarrow \mathbb{R}$, where $\mathcal{G} = \{g:(0,\bar{x}) \rightarrow \mathbb{R} \; | \; \int_{0}^{\bar{x}} g(y) y^{\alpha-1} < \infty \}$, as 
\begin{equation*}
L(g) \defeq \int_{0}^{\bar{x}} g(y) y^{\alpha-1} \left(\bar{x} - y \right)^{\alpha c} e^{-\phi(\lambda) y} dy.
\end{equation*}
We have
\begin{eqnarray*}
\frac{\partial}{\partial \lambda} A(x,\lambda) &=& \int_{0}^{x} \left[ 1 - \phi^{'}(\lambda) y (\lambda - y) \right] y^{\alpha-1} (x-y)^{\alpha c} e^{-\phi(\lambda) y} dy \\
&=& \int_{0}^{x} \left[ 1 - \lambda \phi^{'}(\lambda) y + \phi^{'}(\lambda) y^{2} \right] y^{\alpha-1} (x-y)^{\alpha c} e^{-\phi(\lambda) y} dy,
\end{eqnarray*}
so 
\begin{equation}
\left[ \frac{\partial}{\partial \lambda} A(x,\lambda) \right]_{x = \bar{x}} = L\left(1 - \lambda \phi^{'}(\lambda) f + \phi^{'}(\lambda) f^{2}\right),
\label{partial_A}
\end{equation}
where $f \in \mathcal{G}$ is such that $f(y) = y$.
On the other hand since $A(\bar{x},\lambda) = 0$, we get
\begin{eqnarray*}
L(\lambda) = L(f) &=& \int_{0}^{\bar{x}} y^{\alpha} (\bar{x}-y)^{\alpha c} e^{-\phi(\lambda) y} dy \\
&=& \int_{0}^{\bar{x}} y^{\alpha} e^{-\phi(\lambda) y} d\left(-\frac{(\bar{x}-y)^{\alpha c+1}}{\alpha c+1}\right) \\
&=& (\alpha c+1)^{-1} \int_{0}^{\bar{x}} \left(\bar{x}-y\right)^{\alpha c+1} d\left(y^{\alpha} e^{-\phi(\lambda) y}\right) \\
&=& (\alpha c+1)^{-1} \int_{0}^{\bar{x}} \left(\bar{x}-y\right) \left( \alpha - \phi(\lambda) y \right) y^{\alpha-1} \left(\bar{x}-y\right)^{\alpha c} e^{-\phi(\lambda) y}  dy \\
&=& (\alpha c+1)^{-1} L\Big(\left(\bar{x}-f\right)\left(\alpha - \phi(\lambda)f\right)\Big)  \\
&=& (\alpha c+1)^{-1} L\Big(\alpha\bar{x} - \alpha f - \phi(\lambda)\bar{x}f + \phi(\lambda)f^{2}\Big),
\end{eqnarray*}
where the second integral is Lebesgue-Stieltjes, and the third integral follows from Lebesgue-Stieltjes integration by parts. So, for $\lambda \in (0,\frac{1}{c+1}) \cup (\frac{1}{c+1},1)$, we get
\begin{eqnarray*}
L(f^{2}) &=& \frac{1}{\phi(\lambda)}\bigg[(\alpha c+1) L(f) - L\Big(\alpha\bar{x} - \alpha f - \phi(\lambda)\bar{x}f \Big)\bigg] \\
&=& \frac{1}{\phi(\lambda)}\bigg[\Big((1+c) \alpha +1  - \frac{c \alpha \bar{x}}{1-\lambda} \Big) L(f)\bigg],
\end{eqnarray*}
where we used the fact that $L(\alpha \bar{x}) = \frac{\alpha \bar{x}}{\lambda} L(\lambda) = \frac{\alpha \bar{x}}{\lambda} L(f)$. Now substituting $L(f^{2})$ in~\eqref{partial_A} yields
\begin{eqnarray*}
\left[\frac{\partial}{\partial \lambda} A(x,\lambda)\right]_{x = \bar{x}} &=& L\Big(\frac{1}{\lambda}f - \lambda \phi^{'}(\lambda) f + \phi^{'}(\lambda) f^{2}\Big) \\
&=& \Bigg(\frac{1}{\lambda} - \lambda \phi^{'}(\lambda) + \frac{\phi^{'}(\lambda)}{\phi(\lambda)}\Big[(1+c)\alpha +1 - \frac{c \alpha \bar{x}}{1-\lambda} \Big] \Bigg) L(f), 
\end{eqnarray*}
which after some tedious but routine computations gives
\begin{eqnarray*}
\left[ \frac{\partial}{\partial \lambda} A(x,\lambda) \right]_{x = \bar{x}} = R(\lambda) \frac{\bar{x} - \Phi(\lambda)}{1-(c+1)\lambda}, \quad \lambda \in \Big(0,\frac{1}{1+c}\Big) \cup \Big(\frac{1}{1+c},1\Big)
\end{eqnarray*}
where $R(\lambda) > 0$, for all $0<\lambda<1$, and
\begin{equation*}
\Phi(\lambda) \defeq \frac{\alpha + (1 - 2 \alpha) \lambda + (\alpha - 1 + \alpha c) \lambda^2}{\alpha \Big((c+1) \lambda^2 - 2\lambda + 1\Big)}.
\end{equation*}
Since $d \Phi(\lambda) / d \lambda = \Big((1-c) \lambda^2 - 2 \lambda + 1 \Big) \big/ \Big(\alpha \big((c+1) \lambda^2 - 2\lambda + 1\big) \Big)^{2},$
we have that $d \Phi(\lambda)/ d \lambda = 0$ at $\lambda = 1 / (1+\sqrt{c})$. Note that the other root, $1 / (1-\sqrt{c})$, falls outside of $(0,1)$ for any $c \geq 1$.
It readily can be seen that $\Phi(\lambda)$ is increasing on $0 < \lambda < \frac{1}{1+\sqrt{c}}$ and decreasing on  $\frac{1}{1+\sqrt{c}} < \lambda < 1$, and so
\begin{equation*}
1    \le  \Phi(\lambda)   \le  \frac{2 \alpha \sqrt{c} + 1}{2 \alpha \sqrt{c}}, \quad \forall 0 \le \lambda \le 1.
\end{equation*}
The differentiability of $\bar{x}(\lambda)$ with respect to $\lambda$ follows from implicit function theorem:
\begin{equation*}
\frac{d \bar{x}(\lambda)}{d \lambda} = -\frac{\frac{\partial}{\partial \lambda} A(\bar{x},\lambda)}{\frac{\partial}{\partial \bar{x}}A(\bar{x},\lambda)},
\end{equation*}
and for that we need to show that $\frac{\partial A(\bar{x},\lambda)}{\partial \bar{x}} \neq 0$ for all $0 < \lambda <1$. If we assume the contrary for some $\lambda$, we get
\begin{eqnarray*}
\alpha c A(\bar{x},\lambda) &=& \alpha c \int_{0}^{\bar{x}} (\lambda - y) y^{\alpha-1} (\bar{x}-y)^{\alpha c} e^{-\phi(\lambda) y} dy = 0,\\
( \bar{x}-\lambda) \frac{\partial}{\partial \bar{x}} A(\bar{x},\lambda) &=& \alpha c \int_{0}^{\bar{x}} (\lambda - y) ( \bar{x}-\lambda) y^{\alpha-1} (\bar{x}-y)^{\alpha c-1} e^{-\phi(\lambda) y} dy = 0,
\end{eqnarray*}
which is impossible since the integrand in the first equality is strictly larger than the one in the second equality: we can see this by looking at the two cases $0 < y < \lambda$ and $\lambda < y < \bar{x}$. From this we can also note that $\frac{\partial}{\partial \bar{x}} A(\bar{x},\lambda) < 0$ for all $0 < \lambda <1$. To see this, first consider the case $\bar{x} > \lambda$, and it follows directly as above that $\frac{\partial}{\partial \bar{x}} A(\bar{x},\lambda) < [\alpha c/( \bar{x}-\lambda)] A(\bar{x},\lambda) = 0$. Now assume that $\bar{x} \leq \lambda$, but since the integrand in the first equality is strictly positive for all $0 < y < \bar{x}$, then $A(\bar{x},\lambda) > 0$ which is impossible. So we get
\begin{equation}
\frac{d \bar{x}(\lambda)}{d \lambda} = S(\lambda) \frac{\bar{x} - \Phi(\lambda)}{1-(c+1)\lambda}, \quad \lambda \in [0,1]
\label{dx_bar}
\end{equation}
where $S(\lambda) > 0$ for all $0 < \lambda < 1$. We also defined $\frac{d \bar{x}(\lambda)}{d \lambda}$ for $\lambda = 0, 1, \hf$ using l'H\^{o}pital's rule (with one-sided differentiability for $\lambda = 0,1$). It is easy to see that 
\begin{eqnarray*}
\bar{x}(0) = \bar{x}(1) = \Phi(0) = \Phi(1) = 1 \quad \text{ and } \quad
\bar{x}\Big(\frac{1}{c+1}\Big) = \Phi\Big(\frac{1}{c+1}\Big) = \frac{(c+1)\alpha + 1}{(c+1)\alpha}.
\end{eqnarray*}
Next we show that $\bar{x}$ is strictly increasing on $(0,\frac{1}{c+1})$. We first show that on this interval, we must have $\bar{x}(\lambda) \geq \Phi(\lambda)$, otherwise there must exist a $\hat{\lambda} \in (0,\frac{1}{c+1})$ such that $\bar{x}(\hat{\lambda}) < \Phi(\hat{\lambda})$. But this contradicts $\bar{x}(\frac{1}{c+1}) = \Phi(\frac{1}{c+1})$ by~\eqref{dx_bar}, increasing property of $\Phi$ and continuity of $\bar{x}$. So $\bar{x}$ is non-decreasing on $(0,\frac{1}{c+1})$. We must also have that $\bar{x}(\lambda) > \Phi(\lambda)$ for $\lambda \in (0,\frac{1}{c+1})$, otherwise if there is a $\hat{\lambda} \in (0,\frac{1}{c+1})$ such that $\bar{x}(\hat{\lambda}) = \Phi(\hat{\lambda})$, then, by~\eqref{dx_bar}, it must be a saddle point of $\bar{x}$. But since $\Phi$ is strictly increasing and $\bar{x}$ is non-decreasing on this interval, this would imply that for an $\veps$ arbitrarily small, we must have $\bar{x}(\hat{\lambda}+\veps) < \Phi(\hat{\lambda}+\veps)$ but this would contradict the non-decreasing property of $\bar{x}$ on this interval by~\eqref{dx_bar}. The same reasoning shows that we must have $\bar{x}(\lambda) < \Phi(\lambda)$ on $(\frac{1}{c+1},\lambda^{*})$ (i.e. $\bar{x}$ is strictly increasing on $(\frac{1}{c+1},\lambda^{*})$) and $\bar{x}(\lambda) > \Phi(\lambda)$ on $(\lambda^{*},1)$ (i.e. $\bar{x}$ is strictly decreasing on $(\lambda^{*},1)$). Now we show that $\lambda^{*} \geq \frac{1}{1+\sqrt{c}}$. For $c=1$ we have $\frac{1}{c+1} = \frac{1}{\sqrt{c}+1}$, hence $\lambda^{*} = \hf$. For $c > 1$, Since $\bar{x}(\lambda)$ is increasing for $0 < \lambda < \lambda^{*}$, decreasing for $\lambda^{*} < \lambda < 1$, and $\bar{x}(\lambda^{*}) = \Phi(\lambda^{*})$, then by~\eqref{dx_bar}, this implies that $\lambda^{*}$ is where the maximum of $\bar{x}(\lambda)$ occurs. Now if we assume that $\lambda^{*} < \frac{1}{1+\sqrt{c}}$, since $\Phi$ is increasing on $(0,\frac{1}{1+\sqrt{c}})$, this would contradict $\bar{x}(\lambda) > \Phi(\lambda)$ on $(\lambda^{*},1)$. Lemma~\ref{lemma_2*} is proved. 
$\blacksquare$
\end{proof}

\input exline

\subsection{Proofs of Theorems~\ref{monotonicity_gamma_theorem} and~\ref{extremal_prob_thm}}
\label{proof_main_theorems} 
We now give the detailed proofs for our main theorems
stated and used in Section~\ref{sec:trace}.

\input exline
\textbf{Proof of Theorem~\ref{monotonicity_gamma_theorem}}

For proving (i),	we first show that $\Delta(x) = 0$ at exactly one point on $\mathbb{R}^{+} = \{ x | x > 0 \}$ denoted by $x(\alpha_{1},\alpha_{2})$. Since $\alpha_{2} > \alpha_{1}$, let $\alpha_{2} = \alpha_{1} + c$, for some $c > 0$. We have
		\begin{eqnarray*}
			\frac{d \Delta(x)}{d x} &=& C(\alpha_{2}) x^{\alpha_{2} - 1 } e^{-\alpha_{2} x} - C(\alpha_{1}) x^{\alpha_{1} - 1 } e^{-\alpha_{1} x} \\
			&=& C(\alpha_{2}) x^{\alpha_{1} - 1 } e^{-\alpha_{1} x} \left( x^{c} e^{-c x} - \frac{C(\alpha_{1})}{C(\alpha_{2})} \right)
		\end{eqnarray*} 
		where $C(\alpha) = (\alpha)^{\alpha}/\Gamma(\alpha)$. The constant $C(\alpha_{1})/C(\alpha_{2})$ cannot be larger than $x^{c} e^{-c x}$, for all $x \in \mathbb{R}^{+}$, otherwise $d \Delta(x)/d x$ would be negative for all $x \in \mathbb{R}^{+}$, and this is impossible since $\Delta(0) = \Delta(\infty) = 0$. The function $x^{c} e^{- c x}$ is increasing on $(0,1)$ and decreasing on $(1,\infty)$, and since $C(\alpha_{1})/C(\alpha_{2})$ is constant, there must exist an interval $(a,b)$ containing $x=1$ such that $d \Delta(x)/d x > 0$ for $x \in (a,b)$ and $d \Delta(x)/d x < 0$ for $x \in (0,a) \cup (b,\infty)$. Now since $\Delta(x)$ is continuous and $\Delta(0) = \Delta(\infty) = 0$, then there must exist a unique $x(\alpha_{1},\alpha_{2}) \in (0,\infty)$ such that $\Delta(x)$ crosses zero (i.e., $\Delta(x) = 0$ at the unique point $x(\alpha_{1},\alpha_{2})$) and that $\Delta(x) < 0$ for $0 < x < x(\alpha_{1},\alpha_{2})$ and $\Delta(x) > 0$ for $x > x(\alpha_{1},\alpha_{2})$.
		
We now prove (ii). The desired inequality 
is equivalent to $\Delta(x) < 0, \forall x < 1$ and $\Delta(x) > 0, \forall x > \big(2  \sqrt{\alpha_{1} ( \alpha_{2} - \alpha_{1}) } + 1 \big) / \big(2 \sqrt{\alpha_{1} ( \alpha_{2} - \alpha_{1}) } \big)$. Without loss of generality consider $\alpha = \alpha_{1}$, and $\alpha_{2} = (1+c) \alpha$, for $c =(\alpha_{2}-\alpha)/\alpha$. 
Define $\tilde{X} \sim Gamma(c \alpha, c \alpha)$ and let $Y(t) = t X_{1} + (1-t) \tilde{X}$. Note that $Y(1) = X_{1}$ and $Y(1/(1+c)) = X_{2}$, so it suffices to show that the CDF of $Y(t)$ is increasing in $t \in [\frac{1}{1+c},1]$ for $x < 1$ and decreasing for $x > (2  \alpha \sqrt{c} + 1)/(2  \alpha \sqrt{c})$. Now, we take the Laplace transform of $Y(t)$ as $$\mathcal{L}[Y(t)](z) = \big(1 + \frac {t z}{\alpha}\big)^{-\alpha} \big(1 + \frac {(1-t) z}{c \alpha}\big)^{-c \alpha} $$ for $Re(z) > \max \left\{-\alpha/t,-c \alpha/(1-t)\right\}$.
The Laplace transform of $F_{Y}$ is
$$\mathcal{L}[F_{Y}](z) = \int_{0}^{\infty} e^{-zx} F_{Y}(x) dx  
= \frac{1}{z} \int_{0}^{\infty} e^{-zx} dF_{Y}(x) 
= \frac{1}{z} \mathcal{L}[Y](z).$$
Note that in the second equality we applied integration by parts and the fact that $F_{Y}(0)=0$. Defining $J(z) \defeq \mathcal{L}[F_{Y}](z)$ and differentiating with respect to $t$ gives
\begin{eqnarray*}
\frac{d J}{d t} = J \frac{d}{d t} \left(\ln(J)\right) 
&=& J \frac{d}{dt} \Bigg(-\ln(z) - \alpha \ln(1 + \frac {t z}{\alpha}) - c \alpha \ln \Big(1 + \frac {(1-t) z}{c \alpha}\Big) \Bigg) \\
&=& \frac{z^{2}}{c \alpha} J \Big((1+c) t -  1\Big) \left(1 + \frac {t z}{\alpha}\right)^{-1} \left(1 + \frac {(1-t) z}{c \alpha}\right)^{-1}.
\end{eqnarray*}
Taking the inverse transform yields
\begin{equation*}
\frac{d }{d t} \Pr \left(Y(t) \leq x \right) = \frac{(1+c) t -  1}{c \alpha} \frac{d^{2}}{dx^{2}} \Pr \left( Y(t) + t \psi_{1} +  \frac{1-t}{c} \psi_{2}   < x \right),
\end{equation*}
where $\psi_{i} \sim Gamma(1,\alpha) \;,\; i=1,2,$ are i.i.d gamma r.v's which are also independent of all $X_{1}$ and $X_{2}$. Now applying Lemma~\ref{lemma_2*} yields the desired results. $\blacksquare$

\input exline
\textbf{Proof of Theorem~\ref{extremal_prob_thm}}
It is enough to prove the theorem for the special case where $\alpha = \beta$ and 
%
the general statement follows from the scaling properties of gamma r.v.

Introduce the random variable $Y \defeq  \sum_{i=1}^{n} \lambda_{i} X_{i}$ with CDF $F_{Y}(x) = \Pr(Y < x)$. As in the proof of Theorem~\ref{monotonicity_gamma_theorem}, define $J(z) \defeq \mathcal{L}[F_{Y}](z) = \frac{1}{z} \mathcal{L}[Y](z)$, 
where $\mathcal{L}[F_{Y}]$ and $\mathcal{L}[Y]$ denote the Laplace transform of $F_{Y}$ and $Y$, respectively and $\mathcal{L}[Y](z) = \prod_{i=1}^{n} \left(1 + \lambda_i z/\alpha \right)^{-\alpha}$ for $Re(z) > -\alpha/\lambda_{i},\; i = 1,2,\ldots,n.$ 

Now consider a vector $\bm{\lambda} \in \Theta$ for which $\lambda_i \lambda_j \neq 0$ for some $i \neq j$. We keep all $\lambda_k, \; k \neq i,j$ fixed and vary $\lambda_{i}$ and $\lambda_{j}$ under the condition that $\lambda_i + \lambda_j = const$. We may assume without loss of generality that $i=1$ and $j=2$.  Vectors for which $\lambda_{i} = 1$ for some $i$, i.e.\ the ``corners'' of $\Theta$, are considered at the end of this proof. Differentiating $J$, we get
\begin{eqnarray}
\frac{d J}{d \lambda_{1}} = J \frac{d}{d \lambda_{1}} \left(\ln J \right) &=& J \frac{d}{d \lambda_{1}} \Big( - \ln(z) - \alpha  \sum_{i=1}^{n} \ln(1+ \frac{\lambda_{i} z}{\alpha}) \Big) \nonumber \\
&=& J \alpha \frac{z^{2}}{\alpha^{2}} \frac {\lambda_1 - \lambda_2}{(1+ \frac{\lambda_{1} z}{\alpha})(1+ \frac{\lambda_{2} z}{\alpha})} \nonumber \\
&=& \frac{1} {\alpha} (\lambda_{1} - \lambda_{2}) z \mathcal{L}[\lambda_{1} \psi_{1}](z) \mathcal{L}[\lambda_{2} \psi_{2}](z) \mathcal{L}[Y](z) 
\label{diff_lap_F}
\end{eqnarray}
where $\psi_{i} \sim Gamma(1,\alpha) \;,\; i=1,2$ are i.i.d gamma r.v's which are also independent of all $X_{i}$'s. 

Letting $W(\lambda) = Y + \lambda_{1} \psi_{1} +  \lambda \psi_{2}$ with the CDF $F_{W(\lambda)}(x)$, it can be shown that since $\lambda_{1} \lambda_{2} \neq 0$, then by Lemma~\ref{lemma_A*}(iii), $F_{W(\lambda)}^{'}(0) = 0, \; \forall \lambda \geq 0$. 
Defining 
\begin{equation}
L(Y, \lambda, x) \defeq {F}_{W(\lambda)}^{''} = \frac{d^{2}}{dx^{2}} \Pr \left( W(\lambda)   < x \right)  = \frac{d^{2}}{dx^{2}} \Pr \Big(Y + \lambda_{1} \psi_{1} +  \lambda \psi_{2}   < x \Big) 
\label{L_def}
\end{equation}
and noting that $\mathcal{L}[W(\lambda)](z) = \mathcal{L}[\lambda_{1} \psi_{1}](z) \mathcal{L}[\lambda \psi_{2}](z) \mathcal{L}[Y](z)$, 
we get
\begin{eqnarray*}
\mathcal{L}\big[L(Y,\lambda, .)\big](z) &=& \int_{0}^{\infty} e^{-zx} L(Y,\lambda, x) dx  \nonumber \\
&=& \int_{0}^{\infty} e^{-zx} F_{W(\lambda)}^{''}(x) dx  \nonumber \\
&=& z \int_{0}^{\infty} e^{-zx} dF_{W(\lambda)}(x) \nonumber \\
&=& z \mathcal{L}\big[W(\lambda)\big](z) \nonumber \\
&=& z \mathcal{L}\big[\lambda_{1} \psi_{1}\big](z) \mathcal{L}\big[\lambda \psi_{2}\big](z) \mathcal{L}\big[Y\big](z).
\end{eqnarray*}
Inverting~\eqref{diff_lap_F} yields
\begin{equation}
\frac{d F_{Y}(x)}{d \lambda_{1}} = \frac {1}{\alpha} (\lambda_{1} - \lambda_{2}) L(Y, \lambda_{2}, x). 
\label{lap_inverted}
\end{equation}

So a necessary condition for the extremum of $F_{Y}(x)$ is either  $\lambda_1\lambda_2 (\lambda_1 - \lambda_2) = 0$  or $L(\lambda_{2}, x)  = 0$. 
Since $\lambda_1 \lambda_2 \neq 0$ then by Lemma~\ref{lemma_A*}, the PDF, $f_{W(\lambda)}(x)$, of the linear form $W(\lambda) = Y + \lambda_{1} \psi_{1} + \lambda \psi_{2}$, for $\lambda>0$, is differentiable everywhere and $f_{W(\lambda)}(0) = 0$. In addition, on the positive half-line, $f_{W(\lambda)}'(x) = 0$ holds at a unique point because $f_{W(\lambda)}(x)$ is a unimodal analytic function (its graph contains no line segment). The unimodality of $f_{W(\lambda)}(x)$ was already proven for all gamma random variables in~\cite[Theorem 4]{szba}.


Now we can prove that, for any $x>0$, if $F_{Y}(x)$ has an extremum then the nonzero $\lambda_i$'s can take at most two different values. Suppose that $\lambda_1\lambda_2 (\lambda_1 - \lambda_2) \neq 0$, then by~\eqref{lap_inverted} we have $L(Y,\lambda_{2}, x) = 0$. Now we show that, for every $\lambda_j \neq 0$,~\eqref{lap_inverted} implies that $\lambda_{i} = \lambda_{1}$ or $\lambda_{i} = \lambda_{2}$. For this, we assume the contrary that $\lambda_{i} \neq \lambda_{1}$, $\lambda_{i} \neq \lambda_{2}$, and by using the same reasoning that led to~\eqref{lap_inverted}, we can show that   
\begin{equation*}
L(Y,\lambda_2, x) = L(Y,\lambda_j, x) = 0
\end{equation*}
for every $\lambda_j \neq 0$, i.e. the point  $x>0$ is simultaneously the mode of the PDF of  $W^{\lambda_2}_{Y}$ and $W^{\lambda_j}_{Y}$ which contradicts Lemma~\ref{lemma_1*}. So we get that $\lambda_{i} = \lambda_{1}$ or $\lambda_2 = \lambda_j$. Thus the extrema of $F_{Y}(x)$ are taken for some $\lambda_1 = \lambda_2 = \ldots = \lambda_k$, $\lambda_{k+1} = \lambda_{k+2} = \ldots = \lambda_{k+m}$, and $\lambda_{k+m+1} = \lambda_{k+m+2} = \ldots = \lambda_{n} = 0$  where  $k+m \le n$, i.e.,
\begin{equation*}
\text{extremum } \Pr \big(\sum_{i=1}^{n}   \lambda_{i}  X_{i} \leq x \big) = \text{extremum }\Pr \big( \frac{\lambda}{k}\sum_{i=1}^{k} X_{i} +  \frac{1-\lambda}{m}\sum_{i=k+1}^{k+m} X_{i} \leq x \big).
\end{equation*}
Here without loss of generality we can assume  $k \ge m \ge 1$. Now the same reasoning as in the end of the proof of~\cite[Theorem 1]{szba} shows an extremum is taken either at $k = m = 1$, or at  $\lambda_1 = \lambda_2 = \ldots = ... = \lambda_{k+m}$. In the former case, by Lemma~\ref{lemma_2*}, for any $x \in (0,1) \cup (\frac{2 \alpha+1}{2 \alpha},\infty)$, the extremum can only be taken at $\lambda \in \{0,\hf,1\}$. However, for any $x \in [1,\frac{2 \alpha+1}{2 \alpha}]$,  in addition to $\lambda \in \{0,\hf,1\}$, the extremum can be achieved for some $\lambda^{*}$ such that $x = \bar{x}(\lambda^{*})$ where $\bar{x}(\lambda)$ denotes the mode of the distribution of $\lambda X_{1} + (1-\lambda) X_{2} + \lambda \psi_{1} +  (1-\lambda) \psi_{2}$. But for such $\lambda^{*}$ and $x$, using~\eqref{lap_inverted} and Lemma~\ref{lemma_2*}(iii) with $\alpha_{1} = \alpha_{2} = \alpha$, one can show that $\Pr(\lambda X_{1} + (1-\lambda) X_{2} \leq x)$ achieves a local maximum. Now including the case where $\lambda_1 = 1$ mentioned earlier in the proof, we get
\begin{eqnarray*}
m_{n}(x) &=& \min_{1 \leq d \leq n} \Pr \left(\frac{1}{d}\sum_{i=1}^{d} X_{i} < x \right) \quad \forall x > 0, \\
M_{n}(x) &=& \max_{1 \leq d \leq n} \Pr \left(\frac{1}{d}\sum_{i=1}^{d} X_{i} < x \right) \quad \forall x \in \Big(0,1\Big) \cup \Big(\frac{2 \alpha+1}{2 \alpha},\infty \Big),
\end{eqnarray*}
where $m_{n}(x)$ and $M_{n}(x)$ are defined in the statement of Theorem~\ref{extremal_prob_thm} in Section~\ref{sec:trace}.
Now applying Theorem~\ref{monotonicity_gamma_theorem} by considering the collection $\alpha_{i} = i \alpha,\; i = 1,2,\ldots,n,$ would yield the desired results. $\blacksquare$

\section{{\sc Matlab} Code}
\label{matlab}

Here we provide a short {\sc Matlab} code, promised in Section~\ref{sec:trace}, to calculate the necessary or sufficient sample sizes to satisfy the probabilistic accuracy guarantees~\eqref{prob_ineq_lower_upper} for a SPSD matrix using the Gaussian trace estimator. This code can be easily modified to be used for~\eqref{prob_ineq} as well.
\begin{lstlisting}
function [N1,N2] = getSampleSizes(epsilon,delta,maxN,r)
% INPUT: 
% @ epsilon:  Accuracy of the estimation .
% @ delta:  Uncertainty of the estimation.
% @ r: Rank of the matrix (Use r = 1 for obtaining the sufficient sample sizes).
% @ maxN: Maximum allowable sample size
% OUTPUT:
% @ N1: The sufficient (or necessary) sample size for (2.2a).
% @ N2: The sufficient (or necessary) sample size for (2.2b).
Ns = 1:1:maxN;
P1 = gammainc(Ns*r*(1-epsilon)/2,Ns*r/2);
I1 = find(P1 <= delta,1,'first');
N1 = Ns(I1); % Necessary/Sufficient sample size obtained for (2.2a).
Ns = (floor(1/epsilon)+1):1:maxN;
P2 = gammainc(Ns*r*(1+epsilon)/2,Ns*r/2);
I2 = find(P2 >= 1-delta,1,'first');
N2 = Ns(I2); % Necessary/Sufficient sample size obtained for (2.2b).
end
\end{lstlisting}


\bibliographystyle{plain}
\bibliography{biblio}

\end{document}

%% file: exline.tex
%
%
\vspace{0.5cm}